\documentclass{article}

\usepackage[margin=3.2cm]{geometry}
\usepackage{setspace}
\usepackage{graphics}

\usepackage{docmute}
\usepackage[utf8]{inputenc}

\usepackage{mathtools}
\usepackage{amsmath}
\usepackage{amsfonts}
\usepackage{amssymb}
\usepackage{amsthm}
\usepackage{aliascnt}

\usepackage{makeidx}
\makeindex
\usepackage{textcomp}
\usepackage[sb]{libertine}
\usepackage[varqu,varl]{zi4}%
\usepackage[libertine,bigdelims,vvarbb]{newtxmath}
\usepackage{bm}
\useosf
\usepackage[linktocpage=true,colorlinks,citecolor=blue,linkcolor=blue,urlcolor=blue, pagebackref]{hyperref}
\usepackage[hyperpageref]{backref}

\usepackage{url}
\usepackage{breakurl}

\usepackage{tikz-cd}

\usepackage{enumitem}
\usepackage[UKenglish]{babel}
\usepackage[makeroom]{cancel}
\definecolor{shadecolor}{rgb}{0.88,0.91,0.95}       % Hintergrundfarbe f<FC>r shaded-Umgebungen aus framed.sty
\usepackage{tcolorbox}

\usepackage{dynkin-diagrams}
\usetikzlibrary{backgrounds}

\usepackage{booktabs}

\usepackage{fancyhdr}
\newcommand{\Z}{\mathbb{Z}}

\newcommand{\R}{\mathbb{R}}
\newcommand{\C}{\mathbb{C}}
\newcommand{\F}{\mathbb{F}}
\renewcommand{\P}{\mathbb{P}}

\newcommand{\cO}{\mathcal{O}}

\newcommand{\Ker}{\operatorname{Ker}}

\renewcommand{\Im}{\operatorname{Im}}

\newcommand{\Hom}{\operatorname{Hom}}
\newcommand{\Fix}{\operatorname{Fix}}

\newcommand{\Aut}{\operatorname{Aut}}

\newcommand{\rank}{\operatorname{rank}}

\renewcommand{\O}{\operatorname{O}}

\newcommand{\Id}{\operatorname{Id}}

\newcommand{\<}{\langle}
\renewcommand{\>}{\rangle}

\newcommand{\tensor}{\otimes}
\newcommand{\Pic}{\operatorname{Pic}}

\renewcommand{\d}{\operatorname{d} \!}

\newcommand\Item[1][]{% 
  \ifx\relax#1\relax  \item \else \item[#1] \fi
  \abovedisplayskip=0pt\abovedisplayshortskip=0pt~\vspace*{-\baselineskip}}

\usepackage[capitalise]{cleveref}

\numberwithin{equation}{section}

\newtheorem{proposition}[equation]{Proposition}
\crefname{proposition}{Proposition}{Propositions}

\newtheorem{lemma}[equation]{Lemma}
\crefname{lemma}{Lemma}{Lemmas}

\crefname{corollary}{Corollary}{Corollaries}

\newtheorem*{corollary*}{Corollary}
\crefname{corollary*}{Corollary}{Corollaries}

\newtheorem{theorem}[equation]{Theorem}
\crefname{theorem}{Theorem}{Theorems}

\crefname{conjecture}{Conjecture}{Conjectures}

\newtheorem*{theorem*}{Theorem}
\crefname{theorem*}{Theorem}{Theorems}

\crefname{claim}{Claim}{Claims}

\theoremstyle{remark}

\crefname{question}{Question}{Questions}

\newtheorem{definition}[equation]{Definition}
\crefname{definition}{Definition}{Definitions}

\newtheorem{example}[equation]{Example}
\crefname{example}{Example}{Examples}

\newtheorem{remark}[equation]{Remark}
\crefname{remark}{Remark}{Remarks}

\crefname{assumption}{Assumption}{Assumptions}

\author{Dino Festi, Daniel Platt, Ragini Singhal, and Yuuji Tanaka}
\date{\today}
\title{Examples of real stable bundles on K3 surfaces}

\newcommand{\Cl}{\operatorname{Cl}}

\graphicspath{}

\def\O{\mathcal{O}}

\usepackage{pgfplots}
\pgfplotsset{compat=1.15}
\usetikzlibrary{arrows}

\begin{document}

\maketitle

\begin{abstract}
    Motivated by gauge theory on manifolds with exceptional holonomy, we construct examples of stable bundles on K3 surfaces that are invariant under two involutions: one is holomorphic; and the other is anti-holomorphic.
    These bundles are obtained via the monad construction, and stability is examined using the Generalised Hoppe Criterion of Jardim--Menet--Prata--S\'{a} Earp, which requires verifying an arithmetic condition for elements in the Picard group of the surfaces.
    We establish this by using computer aid in two critical steps:  
    first, we construct K3 surfaces with small Picard group---one branched double cover of $\mathbb{P}^1 \times \mathbb{P}^1$ with Picard rank $2$ using a new method which may be of independent interest; and 
    second, we verify the arithmetic condition for carefully chosen elements of the Picard group, which provides a systematic approach for constructing further examples.
\end{abstract}

\tableofcontents

\section{Introduction}

Stable vector bundles in algebraic geometry give rise to Hermite--Einstein connections, which are analytic objects that are challenging to construct within differential geometry. This connection is established through the Hitchin–-Kobayashi correspondence, as proven by Donaldson \cite{Donaldson1985, Donaldson1987} and Uhlenbeck--Yau \cite{Uhlenbeck-Yau1989}.
These Hermite--Einstein connections are also important in physics, see for example \cite{Strominger1986}.

In applications, it is sometimes the case that the underlying manifold admits discrete symmetries, and one looks for bundles to which these symmetries can be lifted.
The simplest case is that of involutions.
On a complex manifold, one may, in particular, consider anti-holomorphic and holomorphic involutions.
The former was studied in \cite{Wang1993}.

Stable bundles are also key building blocks for the construction of $Spin(7)$-instantons on compact $Spin(7)$-manifolds in \cite{Tanaka2012} and $G_2$-instantons on compact $G_2$-manifolds in \cite{Earp2015,Jardim2017,Platt2024}, due to the correspondence mentioned above.

Since our motivation for this paper comes from the construction of $G_2$-instantons in \cite{Platt2024} by the second-named author, we briefly review it below:

If $X$ is a K3 surface with holomorphic involution $\iota: X \rightarrow X$ and anti-holomorphic involution $\sigma: X \rightarrow X$, then one may extend them to maps $\iota', \sigma': T^3 \times X \rightarrow T^3 \times X$ in such a way that there exists a resolution of singularities $N^7 \rightarrow (T^3 \times X)/\< \iota', \sigma' \>$, such that $N$ admits 
a Ricci-flat metric.
More precisely, $N$ admits a torsion-free $G_2$ structure.
The manifold $N$ was constructed in \cite{Joyce2017}.
In \cite[Section 5.2]{Platt2024}, one example of a $G_2$-instanton on $N$ was constructed as follows:

\begin{theorem}[\cite{Platt2024}]
If $E \rightarrow X$ is a stable bundle that is infinitesimally rigid and admits a holomorphic involution $\hat{\iota}: E \rightarrow E$ covering $\iota$ and an anti-holomorphic involution $\hat{\sigma}: E \rightarrow E$ covering $\sigma$, then there exists a bundle $E' \rightarrow N^7$ constructed from $E$ that admits a $G_2$-instanton.
\end{theorem}

We do not discuss the details here, but the differential geometric construction above is very explicit by means of \emph{gluing}, and many properties of $E'$ and its instanton can be computed from data of $E$ in a straightforward way.

The assumption of \emph{infinitesimal rigidity} is very restrictive.
In other similar situations, it turned out to be possible to remove this assumption, and instead work with a \emph{Kuranishi map}.
The same is hoped to work in this case, and because of this, constructing non-rigid stable bundles is also of interest.

The goal of this paper is to give explicit examples of stable bundles which admit these involutions.
The methods to construct these bundles are known and we prove no new theorems about these objects.
However, our examples are new, and for many of them we provide the easy-to-use computer code that we wrote to find these examples.
We construct the following completely explicit examples using the \emph{monad construction} for holomorphic vector bundles:

\begin{theorem}
    \label{theorem:main-theorem}
    All K3 surfaces $X$ in the following three items admit one holomorphic involution $\iota: X \rightarrow X$ and one anti-holomorphic involution $\sigma: X \rightarrow X$.
    The bundles $K, K_s$ and $K'$ in the following three items admit a holomorphic involution $\hat{\iota}$ covering $\iota$ and admit an anti-holomorphic involution $\hat{\sigma}$ covering $\sigma$.
    \begin{enumerate}
        \item 
        There exists a K3 surface $X$ that is a branched double cover of $\P^2$ with projection map $\pi:X \rightarrow \P^2$ such that the following is true:
        every infinitesimally rigid stable vector bundle of rank $2$ is isomorphic to $\pi^*(T^* \P^2)$ twisted by a line bundle.

        Furthermore, for $s \in \Z, s>0,$ there exists a stable bundle $K_s \rightarrow X$ of rank $2$ with $c_1(K_s)=c_1(\pi^* \mathcal{O}_{\P^2}(-s))$ and $c_2(K_s)=2s^2$, which is infinitesimally rigid for $s=1$.

        \item 
        There exists a K3 surface $X$ that is a branched double cover of  $ \ \P^1 \times \P^1$ with projection map $\pi:X \rightarrow \P^1 \times \P^1$ such that the following is true:
        there exist stable bundles $K,K' \rightarrow X$ of rank $3$ with $c_1(K)=\pi^* \mathcal{O}_{\P^1 \times \P^1}(-4,-4)$ and $c_2(K)=24$ and $c_1(K')=(-4,-4)$ and $c_2(K')=16$ and $K$ is infinitesimally rigid.

        \item 
        There exists a quartic K3 surface $X \subset \P^3$ and a stable vector bundle $E \rightarrow X$ of rank $2$ such that $c_1(E)=\mathcal{O}_{\P^3}(-3)|_X$ and $c_2(E)=12$.
        The bundle $E$ admits a lift $\hat{\sigma}$ of $\sigma$, and $K=E \oplus \iota^* E$ admits lifts $\hat{\iota}$ and $\hat{\sigma}$ of $\iota$ and $\sigma$ respectively.
    \end{enumerate}
\end{theorem}

The bundles from the first, second, and third items of the theorem are constructed in  \cref{subsection:bundles-on-cp2}, \cref{subsection:P1xP1-examples}, and \cref{subsection:example-on-quartic}, respectively.

The construction of all our examples relies on the following:
we focus on K3 surfaces with small Picard groups, i.e., K3 surfaces that admit not many line bundles.
We find suitable examples using \texttt{Magma}.
Roughly speaking, checking stability of bundles becomes computationally easy, because in these cases there are not many line bundles that can destabilise a given bundle.
This is made rigorous by the Generalised Hoppe Criterion \cite[Theorem 3]{Jardim2017}, that we employ.
We use \texttt{Macaulay2} to check a numerical condition on finitely many carefully chosen line bundles, and can then check this condition by hand on all remaining line bundles, and are able to conclude stability of our constructed bundles in this way.
All computer code used for the examples in the paper can be found at \url{https://github.com/raginisinghalmath/Code-for-stable-bundles}.

\medskip 
\noindent
\textbf{Acknowledgments.}
The authors thank Simon Donaldson, Marcos Jardim, Wim Nijgh, Richard Thomas, and Luya Wang for helpful conversations.
They also thank the anonymous referee for suggesting a proof without computer assistance for \cref{proposition:K-on-P1xP1-stable}.
D.F. is a member of the UMI and the INdAM group GNSAGA; he is grateful to 
P{\i}nar K{\i}l{\i}\c{c}er and Bert van Geemen for their help with computational power. 
D.P. was partially supported by the Simons Collaboration on Special Holonomy in Geometry, Analysis and Physics and partially supported by the Eric and Wendy Schmidt AI in Science Postdoctoral Fellowship, funded by Schmidt Sciences. 
R.S. is funded by the Deutsche Forschungsgemeinschaft (DFG, German Research Foundation) under Germany's Excellence Strategy EXC 2044 –390685587, Mathematics Münster: Dynamics–Geometry–Structure.
Y.T. was partially supported by JSPS Grant-in-Aid for Scientific Research numbers 20H00114, 21H00973, 21K03246, 23H01073, Startup Grant at BIMSA, and the Beijing NSF Beijing International
Scientist Project IS25031. 
Y.T. thanks the Department of Mathematics in Kyoto University
and Beijing Institute of Mathematical Sciences and Applications (BIMSA) for their hospitality. 

\section{Background}

\subsection{K3 surfaces}
In this subsection, we introduce K3 surfaces and the basic notions that will be used later;
we refer to \cite[Chapter 1]{Huybrechts2016} for proofs and much more information.
In this paper, we will only consider K3 surfaces defined over $\C$ and will furthermore only deal with \emph{projective} ones, that is, K3 surfaces that admit an embedding into a projective space.
In other words, our K3 surfaces will always be projective varieties of dimension $2$.

We define a K3 surface as a smooth surface $X$ with trivial canonical divisor $K_X \cong \mathcal{O}_{X}$ and trivial first cohomology group of the structure sheaf $H^1(X,\cO_X)=0$.
We define: 
\begin{align*}
    \text{the \emph{Picard group}:}
    \quad
    \Pic X
    &=
    \{ \text{invertible sheaves on } X \} / \cong,
    \text{ and}
    \\
    \text{the \emph{divisor class group}:}
    \quad
    \Cl X
    &=
    \{ \text{Weil divisors on $X$} \} / \approx,
\end{align*}
where $\cong$ denotes isomorphism of line bundles and $\approx$ denotes linear equivalence of divisors.
On a smooth complex K3 surface, they coincide, namely, there is a natural isomorphism $\Cl X \cong \Pic X$, given by sending a divisor class $[D] \in \Cl X$ to its corresponding line bundle $\mathcal{O}_X (D) \in \Pic X$, see e.g. \cite[Corollary II.6.16]{Hartshorne2013}.
In what follows, with a slight abuse of notation, we will freely switch among divisors, divisor classes and corresponding line bundles.
The intersection pairing of the divisor group induces a pairing on $\Cl X$.
With this pairing, also called intersection pairing, $\Cl X$ turns out to be an even lattice, i.e., a finitely generated, torsion-free abelian group with a non-degenerate pairing $\Cl X \times \Cl X \to \Z$ such that, for every $D\in \Cl X$ the self-intersection $D^2=D.D$ is an even integer.
The rank of $\Pic X$ is called \emph{the Picard number} of $X$, and it is denoted by $\rho (X)$.
The Hodge-index theorem also shows that the $\Cl X$ is a hyperbolic lattice, that is, its signature is $(1, \rho(X)-1)$.

\begin{definition}\label{r:Gram}
    Choosing a basis for a lattice $N$, we can associate a \emph{Gram matrix} to the bilinear pairing. 
    We introduce the notation~$[a\; b\; c]$ with~$a,b,c\in \Z$ for a lattice of rank 2 and a basis with Gram matrix equal to $\begin{pmatrix}
        a & b\\
        b & c
    \end{pmatrix}$.
    We also introduce the notation $\langle a \rangle$ for a lattice of rank 1 and a basis with Gram matrix equal to $\begin{pmatrix}
        a
    \end{pmatrix}$.
\end{definition}

We denote by $U$ the lattice $[0\; 1\; 0]$.
We denote by $E_8(-1)$ is the unique even unimodular negative-definite lattice of rank $8$.

The second cohomology group $H^2(X,\Z)$, endowed with the cup product, is also a lattice. 
In particular, if $X$ is a K3 surface, then $H^2(X,\Z)$ is isometric to the lattice $U^{\oplus 3} \oplus E_8(-1)^{\oplus 2}$, 
which is called \emph{the K3 lattice} and we denote it by $\Lambda_{K3}$.
The exponential sequence for $X$ shows that $\Pic X$ embeds into $H^2(X,\Z)$ and so it follows that $1\leq \rho (X)\leq 22$.
In fact, as $X$ is a K3 surface over $\C$, more is true: 
$H^2(X,\C)=H^2(X,\Z)\otimes \C$ is endowed with a Hodge structure of weight two,
$H^2(X,\C)=H^{2,0}(X)\oplus H^{1,1}(X)\oplus H^{0,2}(X)$ with $\dim H^{1,1}(X)=20$. 
From the Lefschetz theorem on $(1,1)$ classes it follows that $\Pic X=H^{1,1}(X)\cap H^2(X,\Z)$ and hence $1\leq \rho (X)\leq 20$.

If  $L \in \Pic X$ is an ample line bundle then the pair $(X,L)$ is called a \emph{polarised} K3 surface;
the degree of $(X,L)$ is defined as the (even) integer $L^2\in 2\Z$.

\begin{example}
    The double cover $X$ of $\P^2$ ramified above a smooth sextic is a K3 surface.
    The pull-back $L$ of $\mathcal{O}_{\P^2}(1)$
     is an ample line bundle and its self-intersection is $2$, hence $(X, L)$ is a K3 surface of degree $2$, see~\cref{p:DoublePlaneK3}.
\end{example}

\subsection{Moduli of semistable sheaves}
\label{subsection:moudli-of-ss-sheaves}

In this subsection, we recall some facts about semistable and stable sheaves (see e.g. \cite{Huybrechts2010} for the details).
To this end, let $X$ be a complex projective surface, and let $H$ be an ample divisor of $X$. 

Let $\mathcal{E} \rightarrow X$ be a coherent sheaf. 
We define the \emph{degree} of $\mathcal{E}$ by $\deg_H(\mathcal{E}) := c_1(\mathcal{E}) \cdot [H]$. 
We call $\mu (\mathcal{E}) := \frac{\deg_{H} (\mathcal{E})}{ \text{rank} (\mathcal{E})}$ the \emph{slope} of $\mathcal{E}$.  

\begin{definition} A torsion-free coherent sheaf $\mathcal{E}$ on $X$ is called \emph{$\mu$-semistable}, if for any proper subsheaf $\mathcal{F}$ of $\mathcal{E}$ the following holds:
\[
\mu (\mathcal{F})  \leq \mu (\mathcal{E}).
\]
We say that $\mathcal{E}$ is \emph{$\mu$-stable} if the strict inequality holds in the above.
\end{definition}

We denote by $\mathcal{O}_X(1)$ the ample line bundle on $X$ and by $p_{\mathcal{E}} (n) := \chi (\mathcal{E} \otimes \mathcal{O}_X (n))$ the {\it Hilbert polynomial} of a coherent sheaf $\mathcal{E}$ on $X$, where $\displaystyle{\chi (\mathcal{E}) := \sum_{i=1}^{2} (-1)^{i} \dim H^{i} (X, \mathcal{E} )}$ is the Euler characteristic of a coherent sheaf $\mathcal{E}$ on $X$. 

\begin{definition} A torsion-free sheaf $\mathcal{E}$ on $X$ is called {\it Gieseker-semistable} if for any proper subsheaf $\mathcal{F}$ the following holds:
$$ \frac{p_{\mathcal{F}}(n)}{\rank \mathcal{F}}  \leq \frac{p_{\mathcal{E}}(n)}{\rank \mathcal{E}} $$
for sufficiently large $n \in \mathbb{N}$. 
We say it is {\it Gieseker-stable}, if the strict inequality holds for sufficiently large $n \in \mathbb{N}$ in the above. 
\end{definition}

Note that the following holds (see e.g. \cite[Lemma 1.2.13]{Huybrechts2010})
$$ \mu\text{-stable} \Rightarrow \text{Gieseker-stable} \Rightarrow \text{Gieseker-semistable} \Rightarrow \mu\text{-semistable}  .$$
Note also that all $\mu$-semistable bundles are $\mu$-stable bundles, if the rank and degree are coprime (see e.g. \cite[Lemma 1.2.14]{Huybrechts2010}. 

We denote by $\mathcal{M}_{X}^{ss} (r, c_1, c_2)$ the moduli space of Gieseker-semistable sheaves $\mathcal{E}$ (resp. $\mathcal{M}_{X}^{s} (r, c_1, c_2)$ the moduli space of Gieseker-stable sheaves) of rank $r$ with $c_1(\mathcal{E}) =c_1$ and $c_2(\mathcal{E}) =c_2$ on a complex projective surface $X$.  
The number
\[
\text{exp.dim}(r,c_1,c_2)
:=
\Delta - (r^2 -1) \chi (\mathcal{O}_{X}),
\]
where $\Delta := 2 r c_2 -(r-1) c_1^2$ is the discriminant, and $\chi(\mathcal{O}_X) = 1 -h^{0,1} +h^{0,2}$, is called the \emph{expected dimension} of the moduli space $\mathcal{M}_{X}^{ss} (r, c_1, c_2)$.

When $X$ is a K3 surface, the moduli space $\mathcal{M}_{X}^{s} (r, c_1 ,c_2)$ is smooth and the dimension is always even, see \cite[p.168]{Huybrechts2010}.
In the case $r=2$ we have that $\dim \mathcal{M}_{X}^{ss} (2, c_1, c_2)=\text{exp.dim}(2,c_1,c_2)=4c_2-c_1^2-6$, see \cite[p.156]{Huybrechts2010}.
For arbitrary rank, Mukai completely described moduli spaces of dimension $\leq 2$ on K3 surfaces, see \cite{Mukai1987}.
In particular, we have the following for zero-dimensional moduli spaces:

\begin{theorem}[Theorem 6.1.6 in \cite{Huybrechts2010}]
\label{theorem:single-point-moduli}
Let $X$ be a K3 surface. Suppose that $\mathcal{M}_X^{s} (r, c_1 ,c_2)$ is non-empty and $\text{exp.dim}(r,c_1,c_2)=0$. 
Then, $\mathcal{M}_X^{ss} (r, c_1 ,c_2)$ contains only a single point, which is represented by a stable locally free sheaf on $X$. 
\label{th:HL2010}
\end{theorem}

Hereafter, when we say stable vector bundles, it means $\mu$-stable vector bundles, unless otherwise stated. 

In later sections, we use the Generalised Hoppe Criterion from \cite{Jardim2017}, to examine whether bundles we construct are stable or not:

\begin{theorem}[Generalised Hoppe Criterion, Theorem 3 in \cite{Jardim2017}]
    \label{theorem:hoppe-criterion}
Let $X$ be an algebraic variety with a polarisation $H$ and $\Pic(X) \cong \mathbb{Z}^l$ for some $l \geq 0$.
Let $E$ be a holomorphic vector bundle of rank $r$ on $X$. 

\begin{enumerate}
    \item 
    If
    \[ H^0 (X, (\Lambda^s E) \otimes L) =0 \]
    for all $L \in \Pic (X)$ and all $1 \leq s \leq r-1$ with $\deg_{H} (L) \leq - s \mu (E)$, then $E$ is stable. 

    \item 
    If $E$ is stable, then
    \[ H^0 (X, E \otimes L) =0 \]
    for all $L \in \Pic (X)$ with $\deg_{H} (L) \leq - s \mu (E)$.
\end{enumerate}
\end{theorem}

\subsection{Monad bundles}

One way to construct holomorphic bundles is using the \emph{monad construction}, that is as the homology or kernel of a complex of vector bundles
\[
0 \rightarrow A \rightarrow B \rightarrow C \rightarrow 0,
\]
which is exact at $A$ and $C$ (but not necessarily at $B$).
Even if the constituents of the complex are easy to understand, one can construct many complicated bundles by the monad construction.
For example, on $\mathbb{P}^2$, it is an application of a theorem by Beilinson that many moduli spaces of stable bundles of rank $2$ can be constructed via the monad construction, see \cite[Chapter 2 Section 3.2]{Okonek2011}.

We begin by recalling how to compute Chern classes of bundles fitting into a short exact sequence:

\begin{proposition}
    \label{proposition:chern-classes-of-monads}
    \leavevmode
    \begin{enumerate}
        \item 
        Let $X$ be a smooth manifold and $0 \rightarrow A \rightarrow B \rightarrow C \rightarrow 0$ be a short exact sequence of complex vector bundles over it.
        Then
        \begin{align*}
            c_1(A)&=c_1(B)-c_1(C),
            \\
            c_2(A)&=c_2(B)-c_1(B)c_1(C)+c_1(C)^2-c_2(C).           
        \end{align*}
        \item 
        On $\mathbb{P}^2$:
        $c_1 \left( \bigoplus_i \mathcal{O}(k_i) \right)=\sum_i k_i$,
        $c_2 \left( \bigoplus_i \mathcal{O}(k_i) \right)=\sum_{i<j} k_i k_j$. 

        \item 
        On $\mathbb{P}^1 \times \mathbb{P}^1$:
        $c_1 \left( \bigoplus_i \mathcal{O}(k_i,m_i) \right)=(\sum k_i, \sum m_i)$,
        $c_2 \left( \bigoplus_i \mathcal{O}(k_i,m_i) \right)=\sum_{i<j} k_i m_j+k_j m_i$.
    \end{enumerate}
\end{proposition}

\begin{proof}
    The Chern polynomial $c_t(E)=1+c_1(E)t+c_2(E)t^2+\dots$ satisfies $c_t(B)=c_t(A)c_t(C)$ and $(c_t(A))^{-1}=(1-c_1(A)+(c_1(A)^2-c_2(A))t^2+\dots$ and the claims follow from this.
\end{proof}

Complex conjugation maps an algebraic variety over $\C$ to itself, if the variety is defined by polynomials with real coefficients.
More generally, being defined over $\R$ is equivalent to the existence of an anti-holomorphic involution by \cite[Proposition 1.3]{Silhol1989}, and we take this as our definition of real structure of a complex manifold:

\begin{definition}
    Let $X$ be a complex manifold.
    An anti-holomorphic involution $\sigma: X \rightarrow X$ is called a \emph{real structure of $X$}.
\end{definition}

On a complex manifold $X$ equipped with a real structure $\sigma$, it makes sense to consider bundles that are compatible with $\sigma$.
The following definition captures this:

\begin{definition}[p.8 in \cite{Hartshorne1978}]
    Let $X$ be a complex manifold with real structure $\sigma:X \rightarrow X$.
    Let $E \rightarrow X$ be a complex vector bundle over $X$.

    A \emph{real structure} on $E$ is an anti-holomorphic map $\hat{\sigma}: E \rightarrow E$ covering $\sigma$ such that $(\hat{\sigma})^2=\Id$.
    A \emph{symplectic structure} is the same, except that $(\hat{\sigma})^2=-\Id$.

    We say that two real or symplectic structures $\hat{\sigma}$ and $\hat{\sigma}'$ are equivalent if $\hat{\sigma}'=\lambda \hat{\sigma}$ for some $\lambda \in \C$, $|\lambda|=1$.
\end{definition}

An example of a complex vector bundle with a real structure is the tangent bundle of a complex manifold $X$ with real structure $\sigma$.
In this case, $\d \sigma: TX \rightarrow TX$ is a real structure.
Other interesting examples come from the following Proposition:

\begin{proposition}
    \label{proposition:real-structure}
    On $\mathbb{P}^{n_1} \times \mathbb{P}^{n_2}$ consider the complex
    \[
        A=\mathcal{O}(a_1,a_1') \oplus \dots \oplus \mathcal{O}(a_k,a_k')
        \overset{a}{\rightarrow}
        B=\mathcal{O}(b_1,b_1') \oplus \dots \oplus \mathcal{O}(b_l,b_l')
        \overset{b}{\rightarrow}
        C=\mathcal{O}(c_1,c_1') \oplus \dots \oplus \mathcal{O}(c_m,c_m').
    \]
    Viewing the maps $a$ and $b$ as polynomials in the coordinates of $\mathbb{P}^{n_1}$ and $\mathbb{P}^{n_2}$, we have the following:
    \begin{enumerate}
        \item 
        If the map $b$ is defined by polynomials with real coefficients, then the kernel monad $\Ker(b)$ admits a real structure.

        \item
        If both maps $a$ and $b$ are defined by polynomials with real coefficients, then the homology monad $\Ker(b)/\Im(a)$ admits a real structure.
    \end{enumerate}
\end{proposition}

\begin{proof}
    Consider first the case of $n_2=0$, i.e., $\mathbb{P}^n := \mathbb{P}^{n_1} \times \mathbb{P}^{n_2}$ and
    \[
        A=\mathcal{O}(a_1) \oplus \dots \oplus \mathcal{O}(a_k)
        \overset{a}{\rightarrow}
        B=\mathcal{O}(b_1) \oplus \dots \oplus \mathcal{O}(b_l)
        \overset{b}{\rightarrow}
        C=\mathcal{O}(c_1) \oplus \dots \oplus \mathcal{O}(c_m).
    \]

    For any $k \in \mathbb{Z}$ and $p \in \mathbb{P}^n$ let $s$ be a local section of $\mathcal{O}(k)$, i.e., a rational function of degree $k$, with real coefficients and such that $s(p) \neq 0$.
    Let then $\xi$ be the map defined by
    \begin{align*}
        \xi: \mathcal{O}(k)_p & \rightarrow \mathcal{O}(k)_{\overline{p}}
        \\
        s(p) &\mapsto s(\overline{p})
    \end{align*}
    and requiring that it is complex anti-linear.
    This is independent of the choice of $s$, if $t$ is another section and $t(p)=\lambda \cdot s(p)$ for $\lambda \in \C$, then $t(\overline{p})=\overline{\lambda} \cdot s(\overline{p})$, because $s,t$ have real coefficients.
    Furthermore, $\xi$ covers the complex conjugation map on $\mathbb{P}^n$ and satisfies $\xi^2=\Id$.
    The definition of $\xi$ extends to the direct sums $A,B,C$.
    Thus, it remains to check that $\xi$ preserves $\Ker b$ and $\Im a$:

    \begin{enumerate}
        \item 
        On an affine patch $U_i = \{x \in \mathbb{P}^n: x_i=1\} \subset \mathbb{P}^n$ fix $p \in U_i$.
        Let $v=(v_1,\dots,v_l)^T \in \Ker b_p \subset B_p$ and $v_i=\lambda_i \cdot s_i(p)$ for $\lambda_i \in \C$ and local sections $s_i$ of $\mathcal{O}(b_i)$ for $i \in \{1,\dots,l\}$ with the property that they, viewed as homogeneous rational functions, have real coefficients.
        Thus we can view $v_i \in \C$, and $v \in \Ker b_p$,  then this means that
        \begin{align}
            \label{equation:b(p)(v)=0}
            b(p) \cdot v=0 \in \C^m,
        \end{align}
        where $\cdot$ denotes the ordinary matrix multiplication of $b(p) \in \C^{m \times l}$ and $v \in \C^l$.
        Then
        \[
        b(\overline{p}) \cdot \xi(v)
        =
        b(\overline{p}) \cdot \xi( (\lambda_i s_i(p))_{1 \leq i \leq l})
        =
        b(\overline{p}) \cdot (\overline{\lambda_i} s_i(\overline{p}))_{1 \leq i \leq l}
        =
        \overline{b(p)} \cdot (\overline{\lambda_i} \overline{s_i(p)})_{1 \leq i \leq l}
        =
        \overline{b(p) \cdot v}
        =0,
        \]
        where in the second step we used that $\xi$ is complex anti-linear;
        and in the third step we used that $b$ and $s_i$ have real coefficients;
        and in the last step we used \cref{equation:b(p)(v)=0}.

        \item 
        Similar to before, let $w=(w_1,\dots,w_k)^T \in A_p$.
        Now write also $a=(a_{ij})$, where the matrix entries $a_{ij}$ are homogeneous polynomials.
        Then
        \begin{align*}
            \xi (a(p) \cdot w)
            &=
            \xi
            \left(
            \left(
            \sum_{1 \leq j \leq k}
            a_{ij}(p)
            \cdot w_j
            \right)_i
            \right)
            \\
            &=
            \left(
            \xi
            \left(
            \sum_{1 \leq j \leq k}
            a_{ij}(p)
            \cdot w_j
            \right)
            \right)_i
            \\
            &=
            \left(
            \sum_{1 \leq j \leq k}
            \overline{a_{ij}(p)} \cdot \xi(w_j) 
            \right)_i
            \\
            &=
            a(\overline{p}) \cdot \xi(w),
        \end{align*}
        where in the third step we used that $\xi$ is complex anti-linear,
        and in the last step we used that the matrix entries $a_{ij}$ are polynomials with real coefficients.
        Thus, $v=a(p) \cdot w \in \Im (a(p))$ implies $\xi(v) \in \Im (a(\overline{p}))$, which proves the claim. \qedhere
    \end{enumerate}
    
\end{proof}

\subsection{Stable bundles on covering spaces}
\label{subsection:stable-bundles-on-covering-spaces}

Let $X$ be a K3 surface, 
$\iota: X \rightarrow X$ be a holomorphic involution 
and $\sigma: X \rightarrow X$ be  anti-holomorphic involution commuting with $\iota$.
We can view $X$ as a branched double cover over the space $X/\<\iota\>$, where the branch locus is $\Fix(\iota)$.
Later we will construct stable bundles on the base space $X/\<\iota\>$ and pull them back to $X$.
The following two results give two sufficient criteria for this pullback to be stable:

\begin{proposition}[Lemma 9.1.9 in \cite{Donaldson1990}]
\label{proposition:donaldson-kronheimer-stable-pullbacks}
    Let $\pi: X \rightarrow \P^2$ be a branched double cover with non-empty branch locus.
    Let $E \rightarrow \P^2$ be a rank $2$ vector bundle that is stable with respect to $\mathcal{O}_{\P^2}(1)$ on $\P^2$.
    Then $\pi^* E$ is stable with respect to the polarisation $\pi ^* \mathcal{O}_{\P^2}(1)$ on $X$.
\end{proposition}

\begin{proposition}
    \label{proposition:stable-on-pullback}
    Let $\pi: X \rightarrow Y$ be a branched cover of a smooth projective surface.
    Assume that the map $\pi^* : \Pic Y \rightarrow \Pic X$ is an isomorphism.
    Then the following is true: let $E \rightarrow Y$ be a holomorphic vector bundle of any rank satisfying the condition from part 1 of \cref{theorem:hoppe-criterion} with respect to the polarisation $H$ on $Y$.
    Then $\pi^* E$ is stable with respect to the polarisation $\pi ^* H$ on $X$.
\end{proposition}

\begin{proof}
    By assumption we have that
    \[ H^0 (X, (\Lambda^s E) \otimes L) =0 \]
    for all $L \in \Pic (Y)$ and all $1 \leq s \leq r-1$ with $\deg_{H} (L) \leq - s \mu_H (E)$.

    Now let $B \in \Pic(X)$ and $1 \leq s \leq r-1$ with $\deg_{\pi^*H} (B) \leq - s \mu_{\pi^*H} (\pi^*E)$.
    We now show that 
    \[H^0 (X, (\Lambda^s \pi^*E) \otimes B) =0.\]
    
    By assumption, $B=\pi^* L$ for some $L \in \Pic(Y)$, and
    \begin{align}
        \label{equation:pullback-of-degree}
        \deg_{H} (L)
        =
        \frac{1}{2}
        \deg_{\pi^*H} (\pi^* L)
        \leq
        \frac{1}{2}
        \left(
        - s \mu_{\pi^*H} (\pi^*E)
        \right)
        =
        - s \mu_{H} (E).
    \end{align}
    The first equality holds because
    \[
        \deg_{\pi^*H} (\pi^* L)
        =
        c_1(\pi^* L) \cdot [\pi^* H]
        =
        \pi^* c_1(L) \cdot [\pi^* H]
        =
        2 c_1(L) \cdot [H]
        =
        2 \deg_H(L).
    \]
    The last equality in \cref{equation:pullback-of-degree} is checked analogously.
    The inequality in the middle of \cref{equation:pullback-of-degree} is precisely the assumption for $B$.
    Thus, by \cref{theorem:hoppe-criterion} we have that $\pi^* E$ is stable with respect to $\pi^* H$.
\end{proof}

\section{K3 surfaces with small Picard group}

The surjectivity of the period map (see \cite[Theorem 7.4.1]{Huybrechts2016}) for K3 surfaces ensures that every hyperbolic sublattice of $\Lambda_{K3}$ is realised as the Picard lattice of a K3 surface, see \cite[Corollaries 2.9 and 2.10]{Morrison84}.
Unfortunately, the proof is not constructive and in general, given a hyperbolic sublattice  $N$ of  $\Lambda_{K3}$, it is not easy to exhibit a geometric construction of a K3 surface $X$ such that $\Pic X \cong N$.

Since we are interested in K3 surfaces with low Picard numbers but not further specifications about the lattice, we focus on the other features requested by the construction above: the existence of a holomorphic involution and a real structure.
The easiest way to ensure the existence of a holomorphic involution on a K3 surface is to realise it as a double cover of another surface.
There are two classical such constructions: 
\begin{itemize}
    \item double covers of $\P^2$ ramified above a smooth sextic curve; and 
    \item double covers of $\P^1\times \P^1$ ramified above a smooth curve of bidegree $(4,4)$. 
\end{itemize}
Since $\P^2$ as well as $\P^1\times \P^1$ can be defined over  $\R$, in order to ensure the real structure on $X$ it is enough to have a branch locus also defined over $\R$.
In the next subsections, we review the theory of these two constructions and provide some explicit examples.
While the first construction is very well known and there is an abundant literature about it, 
the second construction, although classical, is encountered less frequently and for this reason, we will give more details about it.

\subsection{A branched double cover of \texorpdfstring{$\P^2$}{P2}}
\label{subsection:branched-P2}

Let $B\subset \P^2$ be a smooth plane curve of degree $6$ and define $X$ to be the double cover of $\P^2$ ramified above~$B$.
In this way $X$ can be viewed as the surface in $\P (1,1,1,3)$ defined by 
$$
X\colon w^2 = f(x,y,z),
$$
where $w$ is the variable of weight $3$ and $f$ is the polynomial defining $B$ in $\P^2$.
Let $L$ be a line in $\P^2$ and $H:=\pi^*L\subset X$ its pull-back to $X$.
Then $H$ is a double cover of $L$ ramified above the points $L \cap B$.
As $B$ has degree $6$, the curve $H$ is a double cover of a line ramified above 6 points, i.e., $H$ is a curve of genus $2$.

\begin{proposition}\label{p:DoublePlaneK3}
    $X$ is a K3 surface. If $\rho (X)=1$ then $\Pic X\cong \langle 2 \rangle $ is generated by the class of $\mathcal{O}(H)$.
\end{proposition}
\begin{proof}
    This is a very classical result. As $B$ is smooth, so is $X$.
    For a proof for $K_X \cong \mathcal{O}_{X}$ and $H^1(X,\cO_{X})=0$ see~\cite[Section V.22]{Bar2015}.
    If  we consider $H$ as above, then the adjunction formula on $X$ yields $H^2=2g(H)-2=2$.
    As $\Pic X$ is an even lattice and by hypothesis, it is of rank~$1$, thus, we conclude that $\Pic X\cong \langle 2 \rangle$.
\end{proof}

\begin{remark}
    Using the Veronese embedding of weighted degree~3 of $\P (1,1,1,3)$ inside $\P^{10}$, one can view $X$ as a surface in $\P^{10}$.
\end{remark}

\begin{example}
    \label{example:k3-with-rank-Pic=1}
    The first example of a K3 surface given by a double cover of $\P^2$ and having Picard number~$1$ was given in~\cite{EJ08}.
    Nowadays, it is easy to construct many more examples, see e.g.~\cite[Example 3.4]{FNP24}.
\end{example}

\subsection{A branched double cover of \texorpdfstring{$\P^1 \times \P^1$}{P1 x P1}}
\label{subsection:branched-P1xP1}

Let $B$ be a smooth curve of bidegree $(4,4)$ in $\P^1\times \P^1$ and define $X$ to be the double cover of $\P^1\times \P^1$ ramified above $B$.

\begin{lemma}
    $X$ is a K3 surface.
\end{lemma}
\begin{proof}
    To prove that $X$ is a K3 surface we need to show that $X$ is smooth, its canonical divisor is trivial, and $H^1(X,\cO_X)=0$. 
    First, notice that $\P^1\times \P^1$ is smooth and that the singularities of a double cover of a smooth variety are determined by the singularities of the branch locus.
    As $B$ is smooth, this means that $X$ is smooth as well.
    To show that $K_X \cong \mathcal{O}_X$ and $H^1(X,\cO_X)=0$ one can use the theory of double coverings developed in \cite[Section V.22]{Bar2015}; 
    in particular, see the beginning of~\cite[Section V.23]{Bar2015}.
\end{proof}

Let $\pi\colon X\to \P^1\times \P^1$ be the two-to-one projection and let $\iota\in \Aut X$ denote the induced involution.
We now describe $\Pic X$ of $X$.

We start by recalling that $\P^1\times \P^1$ is a del Pezzo surface of degree $8$ and its Picard group is generated by the equivalence classes of two line bundles $\mathcal{O}(D_1), \mathcal{O} (D_2)$, where $D_1, D_2$ are divisors defined by
$D_1 := \{ p_1 \} \times \P^1 $ and $D_2 := \P^1 \times \{ p_2 \}$ for some points $p_1,p_2\in\P^1$.
It is easy to compute the intersection numbers $D_1^2=D_2^2=0$ and $D_1.D_2=1$.
For $i=1,2$, denote by $E_i\in \text{Cl} (X)$ the pull-back of $D_i$.
\begin{lemma}\label{l:U2}
    The sublattice of $\ \Cl X \cong \Pic X$ generated by $E_1$ and $E_2$ is isomorphic to $U(2)=[0\; 2\; 0]$.
\end{lemma}
\begin{proof}
    The curve $E_i$ is the double cover of the line $D_i$ ramified above above $D_i\cap B$. 
    As $B$ is smooth of bidegree~$(4,4)$, for a generic choice of $D_i$ one has $\#(D_i\cap B)=4$.
    This means that $E_i$ is an elliptic curve, i.e., $g(E_i)=0$ and hence, using the adjunction formula on $X$, it follows that $E_i^2=0$.

    The intersection points in $E_1\cap E_2$ are the pullback of the intersection points of $D_1\cap D_2 = \{ (p_1,p_2) \}$.
    For a generic choice of $p_1, p_2$, the point  $(p_1,p_2)$ does not lie on $B$.
    It follows that $E_1.E_2=2$, concluding the proof.
\end{proof}

If $R\subset X$ represents the pre-image $\pi^{-1}(B)$, then by construction the fixed locus of $\iota$ is $\Fix \iota = R$.
\begin{lemma}
    The curve $R$ has genus $g(R)=9$.
\end{lemma}
\begin{proof}
    As $B$ is the branch locus of $\pi$, it is isomorphic to its preimage  $R$, hence $g(B)=g(R)$
    The genus of $B$ is determined by the bidegree-genus formula, 
    $g(B)=(4-1)(4-1)=9$.
\end{proof}

\begin{lemma}
    We have $R=2E_1+2E_2$ in $\Cl X$.
\end{lemma}
\begin{proof}
    As $R$ has genus $9$, its self-intersection is $R^2=2 \cdot 9 -2=16$.
    We have already seen that the intersection $E_i\cap B$ contains $4$ points, because $B$ has bidegree $(4,4)$.
    As $E_i$ and $R$ are the pull-backs of $D_i$ and $B$ respectively,  and $B$ is the branch locus of $\pi$, it follows that $E_i.R=4$.
    Hence the sublattice generated by $E_1,E_2$ and $R$ has Gram matrix
    \[
    \begin{pmatrix}
        0 & 2 & 4\\
        2 & 0 & 4\\
        4 & 4 & 16
    \end{pmatrix}\; .
    \]
    This matrix has the determinant equal to $0$ and hence $R\in \langle E_1, E_2\rangle$.
    Looking at the intersection numbers it is immediate to see that $R=2E_1+2E_2$.
\end{proof}
\begin{proposition}
    Assume that $\rho (X)=2$,
    then $\Cl X = \langle E_1, E_2 \rangle \cong U(2) = [0\; 2\; 0]$.
\end{proposition}
\begin{proof}
    \cref{l:U2} tells us that $U(2)$ is a sublattice of $\Cl X$.
    By initial assumption, $\rho (X)=2$ and hence $U(2)$ is a finite-index sublattice of $\Cl X$.
    As $\det U(2)=-4$ we only have two options for $\Cl X$: it is isomorphic to either $U(2)$ or $U$.

    The involution $\iota$ on $X$ has a curve of genus $9$ as fixed locus;
    as symplectic automorphisms of K3 surface only fix finitely many points (see \cite[Proposition 1.2]{Muk88}) we conclude that $\iota$ is non-symplectic. 
    The involution $\iota$ naturally induces a Hodge isometry $\iota^*$ of order two on $\Cl X$.
    The induced involution $\iota^*$ acts as the identity on the sublattice $\langle E_1, E_2 \rangle\subset \Cl X$, as $E_i$ is the double cover of $D_i\subset \P^1\times \P^1$.
    The classification of K3 surfaces with a non-symplectic involution provided by Nikulin in~\cite[\S 4]{Nik83} then tells us that the sublattice of $\Cl X$ fixed by $\iota^*$ is isomorphic to $U(2)$, hence it is exactly $\langle E_1, E_2 \rangle$.

    Assume $\Cl X\cong U$ and let $F_1, F_2$ be two generators of $\Cl X$.
    Without loss of generality we may and do assume that $F_1, F_2$ are effective.
    This means that $F_1^2=F_2^2=0$ and $F_1.F_2=1$. 
    Up to isometries of $U$, there is only one embedding of $U(2)$ into $U$ that preserves effectiveness, namely $E_1=F_1$ and $E_2=2F_2$.
    It follows that $\iota^*$ fixes the whole Picard lattice $U$, contradicting Nikulin's classification.
    We then conclude that $\Cl X \cong U(2)$.
\end{proof}

\begin{remark}\label{r:Embeddings}
    With a complete description of $\Cl X$, it is also possible to understand the ample divisors of~$X$.
    It is easy to see that $U(2)$ does not admit any $-2$-divisor and hence any effective divisor with positive self-intersection is ample, see~\cite[Corollary 8.1.6]{Huybrechts2016}.
    The divisor $F=E_1+E_2$ is effective and has self-intersection $4$, so it is ample. 
    Let $C$ be a (smooth) curve in the linear system $|F|$. Then $C$ has genus~$3$.
    We claim that $C$ is hyperelliptic: indeed it is fixed by the involution $\iota^*$ and $C.R=8$ tells us that $\pi_*(C)$ has genus $0$, i.e., $C$ is a double cover $\P^1$ or, in other words, $C$ is a hyperelliptic curve.
    This means that $|F|$ induces a map $\phi_{F}\colon X\to \P^3$ which is 2:1 onto its image (see~\cite[Remark 2.2.4.ii]{Huybrechts2016}).
    The divisor $2F=2E_1+2E_2=R$ is also effective and with positive self-intersection, and so it is also ample. In fact, it is \emph{very} ample.
    As $g(R)=9$, it induces an embedding of $\phi_{R}\colon X\hookrightarrow\P^9$.
    This embedding can also be described in an explicit way.
    
    Consider the Segre embedding $\sigma \colon \P^1 \times \P^1 \hookrightarrow \P^3$.
    If $x_0, x_1, x_2, x_3$ are the projective coordinate of $\P^3$,
    then the image $V\cong \P^1 \times \P^1$ of $\sigma$ is the quadric defined by $x_0x_3-x_1x_2=0$.
    Let $B'\subset V$ be the push-forward of $B$ in $V$.
    Then $B'\subset \P^3$ is a curve defined by the intersection of $V$ with an hypersurface $S\colon f(x_0,x_1,x_2,x_3)=0$.
    We can then view $X$ as a double cover of $V$ ramified above $B'=S\cap V$, that is, 
    $X$ can be viewed as the following variety in the weighted projective space $\P (1,1,1,1,2)$
    \begin{equation}
    X\cong 
    \begin{cases}
        & x_0x_3-x_1x_2=0\\
        & y^2 = f(x_0,x_1,x_2,x_3)
    \end{cases}
    \subset \P (1,1,1,1,2)
    \end{equation}
    where $x_0,x_1,x_2,x_3$ are the coordinates of weight 1 and $y$ has weight 2.
    This is the map $\phi_{F}$ above.
    The weighted projective space $\P (1,1,1,1,2)$ can be embedded into $\P^{10}$ using the Veronese embedding of (weighted) degree~$2$
    \[ 
    v\colon \P (1,1,1,1,2) \hookrightarrow \P^{10} \; .
    \]
    Under this embedding, the quadric of $\P (1,1,1,1,2)$ defined by $x_0x_3-x_1x_2=0$ is sent to a hyperplane and hence $v(X)$ can be embedded into $\P^9$, that is, we retrieve the embedding $\phi_{R}$.
\end{remark}

\begin{example}
    \label{example:branched-P1xP1-with-small-Picard}
    Let $X$ be the K3 surface obtained as the double cover of $\P^1\times\P^1$ ramified above the curve
    \begin{align*}
        B\colon & 2x_0^3x_1y_0^4 + x_0^4y_0^3y_1 + 2x_0^3x_1y_0^3y_1 + 
                 x0^2x_1^2y_0^3y_1 + x_0x_1^3y_0^3y_1 + 2x_1^4y_0^3y_1 +\\
                & x_0^4y_0^2y_1^2 + 2x_0^3x_1y_0^2y_1^2 + x_0^2x_1^2y_0^2y_1^2 + 
                 2x_1^4y_0^2y_1^2 + x_0^3x_1y_0y_1^3 + x_0^2x_1^2y_0y_1^3 + \\ 
                & x_0x_1^3y_0y_1^3 + x_1^4y_0y_1^3 + 2x_0^4y_1^4 + 2x_0^3x_1y_1^4 \subset \P^1\times\P^1
    \end{align*}
    of bidegree $(4,4)$,
    where $x_0,x_1$ and $y_0,y_1$ are the projective coordinates of the two copies of $\P^1$, respectively.
    By reducing modulo $3$ and computing the number of points over $\F_{3^n}$ for $n=1,\dots,9$, 
    one can prove that $X$ has Picard number $2$.
    See the ancillary \texttt{Magma} file to see how we obtained the equation defining $B$ and the proof that $\rho (X)=2$.
    In the same file, we also provide equations of $X$ as a surface in $\P^{10}$ and $\P^9$, using the embeddings presented in \cref{r:Embeddings}.
\end{example}

\section{Stable bundles on K3 surfaces with small Picard group}
\label{section:stable-bundle-construction}

We will now put everything together to construct several examples of stable bundles on K3 surfaces that are invariant under a holomorphic and an anti-holomorphic involution.
In \cref{subsection:bundles-on-cp2} and \cref{subsection:P1xP1-examples}, we will construct real stable bundles on $\mathbb{P}^2$ and $\mathbb{P}^1 \times \mathbb{P}^1$, respectively.
The pullbacks of these bundles to suitable branched double covers will then give stable bundles on K3 surfaces.
Last, in \cref{subsection:example-on-quartic}, we will write down a real stable bundle on a K3 surface directly.
This will \emph{not} be invariant under a holomorphic involution of the K3 surface, but by taking the direct sum with its pullback one obtains a polystable bundle that is invariant under the holomorphic involution.

\subsection{Examples of bundles on the branched double cover of \texorpdfstring{$\P^2$}{P2}}
\label{subsection:bundles-on-cp2}

The double cover $\pi: X \rightarrow \P^2$ branched over a smooth sextic $f$ is a classical example of a K3 surface that was reviewed in \cref{subsection:branched-P2}.
We have the holomorphic involution $\iota:X \rightarrow X$ swapping the sheets of the branched cover.
If the sextic $f$ is real, then there exists a lift $\sigma:X \rightarrow X$ of the complex conjugation $\P^2 \rightarrow \P^2$.
Importantly, one can choose real sextics so that $\Pic(X)= \pi ^*(\Pic (\P^2))$, this was \cref{example:k3-with-rank-Pic=1}.
For the rest of the section, we assume this is the case.
In this case, real stable bundles on $\P^2$ pull back to real stable bundles on $X$ by \cref{proposition:stable-on-pullback}.

Knowing all this, we will now construct stable bundles on $\P^2$.
Each of these examples corresponds to a stable bundle on a branched double cover of $\P^2$ with the property that $\Pic(X)= \pi ^*(\Pic (\P^2))$.

That said, a whole lot is known about stable bundles on $\P^2$, see \cite{Okonek2011} for an overview of the field.
In \cite[Section 2.2]{Wang1993}, even the moduli space of stable bundles \emph{with real structure} on $\P^2$ was identified in the case of rank $2$ and $c_1=0, c_2=2$.

In this section, we present two examples:
the cotangent bundle of $\P^2$, which is the only rigid stable bundle of rank $2$ on $\P^2$ and
another (non-rigid) rank $2$ bundle.

\subsubsection{The pullback of the cotangent bundle of \texorpdfstring{$\P^2$}{P2}}

The Euler sequence is the short exact sequence
\begin{align}
\label{equation:euler-sequence}
    0
    \rightarrow
    T^* \P^2
    \rightarrow
    \mathcal{O}(-1)^{\oplus 3}
    \rightarrow
    \mathcal{O}(0)
    \rightarrow
    0,
\end{align}
so $T^* \P^2$ is a kernel monad bundle.
The differential of the complex conjugation $\P^2 \rightarrow \P^2$ defines a real structure on $T^* \P^2$.
It is easy to check that $T^* \P^2$ is stable on $\mathbb{P}^2$, so by \cref{proposition:donaldson-kronheimer-stable-pullbacks} we have that $\pi^*(T^* \mathbb{P}^2)$ is also stable.
By construction, it is invariant under the holomorphic and anti-holomorphic involution on~$X$.
It is also easy to check that it is infinitesimally rigid:
this follows from \cref{theorem:single-point-moduli} together with the smoothness of the moduli space on a K3 surface.
Up to twisting by the polarisation, it is the \emph{unique} stable and infinitesimally rigid bundle of rank $2$ on $X$.
This proves the first part of the first point in \cref{theorem:main-theorem}.

\begin{lemma}
\label{lemma:unique-stable-rk2-bundle-up-to-twists}
    Every infinitesimally rigid stable bundle $E \rightarrow X$ of rank $2$ is isomorphic to $\pi^*(T^* \mathbb{P}^2)(k)$ for some $k \in \mathbb{Z}$.
\end{lemma}

\begin{proof}
    Let $c_1(E)=\pi^*\mathcal{O}(x)$ and $c_2(E)=y$.
    By \cref{subsection:moudli-of-ss-sheaves}, we have
    \[
    0
    =\dim \mathcal{M}_X^{s} (2, c_1(E), c_2(E))
    =\text{exp.dim} (2, c_1(E), c_2(E))
    =4c_2(E)-c_1(E)^2-6
    =
    4y-2x^2-6,
    \]
    using $c_1(\pi^*\mathcal{O}(1)) \cup c_1(\pi^*\mathcal{O}(1))=2$ on $X$.
    One easily checks that all solutions of this equation are of the form $x=2k+1$ and $y=2+2k+2k^2$ for $k \in \mathbb{Z}$.
    These are precisely the Chern classes of the bundle $\pi^*(T^* \mathbb{P}^2(k+2))$, i.e.
    \[
    c_1(\pi^* T^* \mathbb{P}^2(k+2))
    =
    \pi^*\mathcal{O}(2k+1)
    \text{ and }
    c_2(\pi^* T^* \mathbb{P}^2(k+2))
    =
    2+2k+2k^2,    
    \]
    which one quickly verifies using $c_1(\pi^*(T^* \mathbb{P}^2))=\pi^*\mathcal{O}(-3)$ and $c_2(\pi^*(T^* \mathbb{P}^2))=6$.
    
    Thus, whenever the dimension of $\mathcal{M}_X^{s} (2, c_1(E), c_2(E))$ is zero, it contains a twist of $\pi^*(T^* \mathbb{P}^2)$.
    By \cref{theorem:single-point-moduli}, this is the only point in the moduli space, hence $E$ is isomorphic to this twist.
\end{proof}

\subsubsection{Another rank \texorpdfstring{$2$}{2} example}

Here is a generalisation of the previous example.
It is still of rank $2$, so by \cref{lemma:unique-stable-rk2-bundle-up-to-twists} it cannot be infinitesimally rigid.
For $s>0$ consider $K$ defined via
\begin{align}
\label{equation:cotangent-bundle-generalisation-complex}
K:=Ker (a) \rightarrow \mathcal{O}(0)^{\oplus 3}
\overset{a=(x^s \,\,\, y^s \,\,\, z^s)}{\joinrel\relbar\joinrel\relbar\joinrel\relbar\joinrel\relbar\joinrel\relbar\joinrel\relbar\joinrel\relbar\joinrel\relbar\joinrel\relbar\joinrel\rightarrow}
\mathcal{O}(s).
\end{align}
Note that for $s=1$ we have that $K \cong T^* \P^2(1)$, which can be seen from tensoring \cref{equation:cotangent-bundle-generalisation-complex} by $\mathcal{O}(1)$ and obtaining \cref{equation:euler-sequence}.
We then have the following proposition, which is the second part of the first point in \cref{theorem:main-theorem}.

\begin{proposition}
    The bundle $K$ defined in \cref{equation:cotangent-bundle-generalisation-complex} is real and stable.
\end{proposition}

\begin{proof}
    The bundle $K$ has a real structure by \cref{proposition:real-structure}.
    We will prove that $K^*$ is stable.
    From \cref{proposition:chern-classes-of-monads} we have that $c_1(K^*)=s$ and $c_2(K^*)=s^2$.
    By \cref{theorem:hoppe-criterion}, we have that $K^*$ is stable if $H^0(K^*(k))=0$ for $k \leq -\frac{s}{2}$.
    Let $k_0=\left\lfloor -\frac{s}{2} \right\rfloor$.
    Taking the dual of \cref{equation:cotangent-bundle-generalisation-complex}, tensoring by $\mathcal{O} \left( k_0 \right)$ and passing to the long exact sequence in cohomology gives
    \[
        \dots
        \rightarrow
        \underbrace{
        H^0 \left( \mathcal{O} \left( k_0 \right)^{\oplus 3} \right)
        }_{=0}
        \rightarrow
        H^0 \left( K^* (k_0) \right)
        \rightarrow
        \underbrace{
        H^1 \left( \mathcal{O} \left(-s +k_0 \right) \right)
        }_{=0}
        \rightarrow
        \dots
    \]
    Thus, $H^0 \left( K^* (k_0) \right)=0$ which implies $H^0 \left( K^* (k) \right)=0$ for $k \leq k_0$.
    Hence, by \cref{theorem:hoppe-criterion}, we have that $K^*$ is stable.
    By \cite[Proposition 1.4 (iv)]{Takemoto1972}, we have that also $K$ is stable, which proves the claim.
\end{proof}

\subsection{Examples of bundles on the branched double cover of \texorpdfstring{$\P^1 \times \P^1$}{P1 x P1}}
\label{subsection:P1xP1-examples}

The double cover $\pi: X \rightarrow \P^1 \times \P^1$ branched over a smooth curve of bi-degree $(4,4)$ is another example of a K3 surface, see \cref{subsection:branched-P1xP1}.
As in the case of branched covers of $\mathbb{P}^2$, if one chooses a \emph{real} branching curve, then $X$ has a real structure.
Again, one can choose branching curves so that $\Pic(X) = \pi^*( \Pic (\P^1 \times \P^1))$, see \cref{example:branched-P1xP1-with-small-Picard}.
Hence, we are led to construct examples of real stable bundles on $\P^1\times \P^1$, as every such bundle corresponds to a real stable bundle on a branched double cover of $\P^1 \times \P^1$ with the property $\Pic(X) = \pi^*( \Pic (\P^1 \times \P^1))$ by \cref{proposition:stable-on-pullback}.

Stable bundles on $\P^1 \times \P^1$ have also been studied previously.
The complex surface $\P^1 \times \P^1$ is isomorphic to a smooth quadric in $\P^3$ and the special case $\mathbb{F}_0$ in the family of Hirzebruch surfaces.
Two classical references studying stable bundles are \cite{Buchdahl1987,Socorro1985}, and a recent account can be found in \cite{Huh2011}.
In particular, \cite{Huh2011} gives an explicit description of some moduli spaces of rank $2$ bundles on $\P^1 \times \P^1$.

In this subsection, we present three examples:
one rank $2$ example that is non-rigid, 
one rigid rank $3$ example, and a family of non-rigid rank $3$ examples.

\subsubsection{A rank \texorpdfstring{$2$}{2} example}

Consider the complex
\begin{align}
\label{equation:complex-for-rank2-ex-on-P1xP1}
A=\mathcal{O}(0,0)
\overset{\begin{pmatrix}
    x_1 & x_2 & y_1 & y_2
\end{pmatrix}}{\longrightarrow}
B=\mathcal{O}(1,0)^{\oplus 2} \oplus \mathcal{O}(0,1)^{\oplus 2}
\overset{\begin{pmatrix}
    y_1 \\ y_2 \\ -x_1 \\ -x_2
\end{pmatrix}}{\longrightarrow}
C = \mathcal{O}(1,1)
\end{align}
and its cohomology monad 
\begin{align}
\label{equation:cohomology-monad-on-P1xP1}
E:= \Ker
\begin{pmatrix}
    y_1\\y_2\\-x_1\\-x_2
\end{pmatrix}
/
\Im (x_1 \,\, x_2 \,\, y_1 \,\, y_2).
\end{align}

\begin{proposition}
    \label{proposition:rank-2-on-P1xP1}
    The bundle $E \rightarrow \mathbb{P}^1 \times \mathbb{P}^1$ defined in \cref{equation:cohomology-monad-on-P1xP1} is real and stable.
\end{proposition}

\begin{proof}
The bundle $E$ has a real structure by \cref{proposition:real-structure}.
Splitting the complex from \cref{equation:complex-for-rank2-ex-on-P1xP1} into two long exact sequences gives:
\begin{align*}
    0 &\rightarrow K \rightarrow B \rightarrow C \rightarrow 0,
    \\
    0 &\rightarrow A \rightarrow K \rightarrow E \rightarrow 0,
\end{align*}
where $K=\Ker \begin{pmatrix}
    y_1\\y_2\\-x_1\\-x_2
\end{pmatrix}$.
Using \cref{proposition:chern-classes-of-monads}, one checks 
\begin{align*}
    c_1(K)&=c_1(B)-c_1(C)=(1,1),\;\; \text{ and }\\ c_1(E)&=c_1(K)-c_1(A)=(1,1) \; .
\end{align*}

Hence, by the Hoppe criterion \cref{theorem:hoppe-criterion}, $E$ is stable if $H^0(E \tensor \mathcal{O}(k,l))=0$ for $k+l \leq -1$.
All pairs $(k,l)$ satisfying $k+l \leq -1$ are shown in \cref{figure:twist-triangle}.
As shown in the figure, these pairs $(k,l)$ fall into two classes that are treated separately.

\begin{figure}[htbp]
\centering
\definecolor{ffqqtt}{rgb}{1,0,0.2}
\definecolor{ududff}{rgb}{0.30196078431372547,0.30196078431372547,1}
\definecolor{xdxdff}{rgb}{0.49019607843137253,0.49019607843137253,1}
\definecolor{qqwwzz}{rgb}{0,0.4,0.6}
\definecolor{wqwqwq}{rgb}{0.3764705882352941,0.3764705882352941,0.3764705882352941}
\begin{tikzpicture}[line cap=round,line join=round,>=triangle 45,x=1cm,y=1cm]
\begin{axis}[
x=1cm,y=1cm,
axis lines=middle,
ymajorgrids=true,
xmajorgrids=true,
xmin=-4.291413200673038,
xmax=4.511485486480677,
ymin=-2.5712705244191194,
ymax=2.5688128469303795,
xtick={-4,-3,...,4},
ytick={-2,-1,...,2},
        xlabel=$k$, xlabel style={alias=aux},
        ylabel=$l$,  ylabel style={alias=auy}]
\clip(-4.291413200673038,-2.5712705244191194) rectangle (4.511485486480677,2.5688128469303795);
\fill[line width=1pt,color=qqwwzz,fill=qqwwzz,fill opacity=0.31] (-1.5,1) -- (-1.5,-1.5) -- (1,-1.5) -- cycle;
\fill[line width=1pt,color=ffqqtt,fill=ffqqtt,fill opacity=0.27] (-6.282802587468886,5.782802587468886) -- (-1.5,1) -- (-1.5,-1.5) -- (1,-1.5) -- (7.496074129017353,-7.996074129017353) -- (-6.551819384040839,-8.036963827634215) -- cycle;
\draw [line width=1pt,domain=-4.291413200673038:4.511485486480677] plot(\x,{(-0.5-1*\x)/1});
\draw [line width=1pt,color=qqwwzz] (-1.5,1)-- (-1.5,-1.5);
\draw [line width=1pt,color=qqwwzz] (-1.5,-1.5)-- (1,-1.5);
\draw [line width=1pt,color=qqwwzz] (1,-1.5)-- (-1.5,1);
\draw [line width=1pt,color=ffqqtt] (-6.282802587468886,5.782802587468886)-- (-1.5,1);
\draw [line width=1pt,color=ffqqtt] (-1.5,1)-- (-1.5,-1.5);
\draw [line width=1pt,color=ffqqtt] (-1.5,-1.5)-- (1,-1.5);
\draw [line width=1pt,color=ffqqtt] (1,-1.5)-- (7.496074129017353,-7.996074129017353);
\draw [line width=1pt,color=ffqqtt] (7.496074129017353,-7.996074129017353)-- (-6.551819384040839,-8.036963827634215);
\draw [line width=1pt,color=ffqqtt] (-6.551819384040839,-8.036963827634215)-- (-6.282802587468886,5.782802587468886);
\begin{scriptsize}
\draw [fill=wqwqwq] (-4,3) circle (2.5pt);
\draw [fill=wqwqwq] (-3,3) circle (2.5pt);
\draw [fill=wqwqwq] (-2,3) circle (2.5pt);
\draw [fill=wqwqwq] (-1,3) circle (2.5pt);
\draw [fill=wqwqwq] (0,3) circle (2.5pt);
\draw [fill=wqwqwq] (1,3) circle (2.5pt);
\draw [fill=wqwqwq] (2,3) circle (2.5pt);
\draw [fill=wqwqwq] (3,3) circle (2.5pt);
\draw [fill=wqwqwq] (4,3) circle (2.5pt);
\draw [fill=wqwqwq] (4,2) circle (2.5pt);
\draw [fill=wqwqwq] (3,2) circle (2.5pt);
\draw [fill=wqwqwq] (2,2) circle (2.5pt);
\draw [fill=wqwqwq] (1,2) circle (2.5pt);
\draw [fill=wqwqwq] (0,2) circle (2.5pt);
\draw [fill=wqwqwq] (-1,2) circle (2.5pt);
\draw [fill=wqwqwq] (-2,2) circle (2.5pt);
\draw [fill=wqwqwq] (-3,2) circle (2.5pt);
\draw [fill=wqwqwq] (-4,2) circle (2.5pt);
\draw [fill=wqwqwq] (-4,1) circle (2.5pt);
\draw [fill=wqwqwq] (-3,1) circle (2.5pt);
\draw [fill=wqwqwq] (-2,1) circle (2.5pt);
\draw [fill=wqwqwq] (-1,1) circle (2.5pt);
\draw [fill=wqwqwq] (0,1) circle (2.5pt);
\draw [fill=wqwqwq] (1,1) circle (2.5pt);
\draw [fill=wqwqwq] (2,1) circle (2.5pt);
\draw [fill=wqwqwq] (3,1) circle (2.5pt);
\draw [fill=wqwqwq] (4,1) circle (2.5pt);
\draw [fill=wqwqwq] (4,0) circle (2.5pt);
\draw [fill=wqwqwq] (3,0) circle (2.5pt);
\draw [fill=wqwqwq] (2,0) circle (2.5pt);
\draw [fill=wqwqwq] (1,0) circle (2.5pt);
\draw [fill=wqwqwq] (0,0) circle (2.5pt);
\draw [fill=wqwqwq] (-1,0) circle (2.5pt);
\draw [fill=wqwqwq] (-2,0) circle (2.5pt);
\draw [fill=wqwqwq] (-3,0) circle (2.5pt);
\draw [fill=wqwqwq] (-4,0) circle (2.5pt);
\draw [fill=wqwqwq] (-4,-1) circle (2.5pt);
\draw [fill=wqwqwq] (-3,-1) circle (2.5pt);
\draw [fill=wqwqwq] (-2,-1) circle (2.5pt);
\draw [fill=wqwqwq] (-1,-1) circle (2.5pt);
\draw [fill=wqwqwq] (0,-1) circle (2.5pt);
\draw [fill=wqwqwq] (1,-1) circle (2.5pt);
\draw [fill=wqwqwq] (2,-1) circle (2.5pt);
\draw [fill=wqwqwq] (3,-1) circle (2.5pt);
\draw [fill=wqwqwq] (4,-1) circle (2.5pt);
\draw [fill=wqwqwq] (-4,-2) circle (2.5pt);
\draw [fill=wqwqwq] (-3,-2) circle (2.5pt);
\draw [fill=wqwqwq] (-2,-2) circle (2.5pt);
\draw [fill=wqwqwq] (-1,-2) circle (2.5pt);
\draw [fill=wqwqwq] (0,-2) circle (2.5pt);
\draw [fill=wqwqwq] (1,-2) circle (2.5pt);
\draw [fill=wqwqwq] (2,-2) circle (2.5pt);
\draw [fill=wqwqwq] (3,-2) circle (2.5pt);
\draw [fill=wqwqwq] (4,-2) circle (2.5pt);
\draw [fill=wqwqwq] (-4,-3) circle (2.5pt);
\draw [fill=wqwqwq] (-3,-3) circle (2.5pt);
\draw [fill=wqwqwq] (-2,-3) circle (2.5pt);
\draw [fill=wqwqwq] (-1,-3) circle (2.5pt);
\draw [fill=wqwqwq] (0,-3) circle (2.5pt);
\draw [fill=wqwqwq] (1,-3) circle (2.5pt);
\draw [fill=wqwqwq] (2,-3) circle (2.5pt);
\draw [fill=wqwqwq] (3,-3) circle (2.5pt);
\draw [fill=wqwqwq] (4,-3) circle (2.5pt);
\draw [fill=xdxdff] (-6.282802587468886,5.782802587468886) circle (2.5pt);
\draw [fill=xdxdff] (7.496074129017353,-7.996074129017353) circle (2.5pt);
\draw [fill=ududff] (-6.551819384040839,-8.036963827634215) circle (2.5pt);
\end{scriptsize}
\end{axis}
\end{tikzpicture}
\caption{Pictorial representation of the proof of \cref{proposition:rank-2-on-P1xP1}: the Hoppe criterion requires verifying that $H^0(E \otimes \mathcal{O}(k,l))=0$ for $k+l \leq -1$. All pairs $(k,l)$ satisfying $k+l \leq -1$ lie in the coloured region. In the blue-shaded region, we check the condition using an explicit calculation. In the red-shaded region,  we have $k \leq -2$ or $l \leq -2$, and in this case there is an alternative argument to verify the condition. However, this alternative argument does not work within the blue region.}
\label{figure:twist-triangle}

\end{figure}

\begin{itemize}
\item
\textbf{Computation of finitely many cohomology groups.}
A \texttt{Macaulay2} calculation shows:
\[
h^0(E \otimes \mathcal{O}(-1,0))=h^0(E \otimes \mathcal{O}(0,-1))=h^0(E \otimes \mathcal{O}(-1,-1))=0.
\]

\item
\textbf{Computation of the remaining cohomology groups.}
Let $k,l \in \Z$ such that $k+l \leq -1$ and $(k,l) \notin \{ (-1,0),(0,-1),(-1,-1)\}$.
Then $k \leq -2$ or $l \leq -2$.
We can assume without loss of generality that $l \leq -2$.
For $x=[0:1] \in \mathbb{P}^1$ we have:
\begin{align*}
    h^0(E \otimes \mathcal{O}(k,l))
    \leq
    h^0(E \otimes \mathcal{O}(k,l) \mid_{\{x\} \times \mathbb{P}^1})
    =
    h^0(E\mid_{\{x\} \times \mathbb{P}^1} \otimes \mathcal{O}(l))
    \leq
    h^0(E\mid_{\{x\} \times \mathbb{P}^1} \otimes \mathcal{O}(-2))
    =0,
\end{align*}
where the last equality is again obtained by a \texttt{Macaulay2} calculation.
\qedhere
\end{itemize}
\end{proof}

Using \cref{proposition:chern-classes-of-monads}, one checks that $c_2(K)=2$ and $c_2(E)=2$.
Thus, the bundle constructed from \cref{equation:cohomology-monad-on-P1xP1} is dual to the bundle with $c_1(E)=(-1,-1)$, $c_2(E)=k$ studied in \cite{Huh2011} for the choice of $k=2$.
The dimension of its moduli space is $4\cdot c_2(E)-5=3$.

\subsubsection{A family of rank \texorpdfstring{$3$}{3} examples}

For $n \geq 1$, define the rank $3$ vector bundle $K_n$ over $\mathbb{P}^1 \times \mathbb{P}^1$ as follows:
\begin{align}
\label{equation:K-P1xP1-def}
K_n:= \Ker \begin{pmatrix}
    x_1^n y_1^n \\
    x_1^n y_2^n \\
    x_2^n y_1^n \\
    x_2^n y_2^n
\end{pmatrix},
\quad
0 \rightarrow 
K_n \rightarrow 
\mathcal{O}(-n,-n)^{\oplus 4}
\rightarrow \mathcal{O}
\rightarrow
0.
\end{align}

\begin{proposition}
    \label{proposition:K-on-P1xP1-stable}
    For $n \geq 1$, the bundle $K_n \rightarrow \mathbb{P}^1 \times \mathbb{P}^1$ defined in \cref{equation:K-P1xP1-def} is real and stable.
    For $n=1$ its pullback under the projection $\pi: X \rightarrow \mathbb{P}^1 \times \mathbb{P}^1$ is infinitesimally rigid.
\end{proposition}

\begin{proof}
    The bundle $K_n$ has a real structure by \cref{proposition:real-structure}.
    
    We next prove the stability of the bundle. First,  assume that $n=1$. 
    Let $s:S=\P^1 \times \P^1 \rightarrow Q \subset \mathbb{P}^3$ be the Segre embedding.
    Pulling back the Euler sequence twisted by $\mathcal{O}(-1)$ gives $K_1 \simeq s^*(\Omega_{\mathbb{P}^3}|_Q)$.
    \cite[Theorem B]{Biswas2019} gives that $\Omega_{\mathbb{P}^3}|_Q$ is $\mathcal{O}(1,1)$-stable, hence $K_1$ is stable.

    Now, let $n \geq 1$. Define $\nu_n([u:v]) :=[u^n:v^n]$ and $\pi_n :=(\nu_n, \nu_n): \P^1 \times \P^1 \rightarrow \P^1 \times \P^1$.
    Pulling back \cref{equation:K-P1xP1-def} shows $K_n \simeq \pi_n^* K_1$.
    By \cite[Corollaire 5.2]{Grothendieck1971}, the étale fundamental group satisfies $\pi^{\text{ét}}_1(\P^1 \times \P^1)=1$.
    By \cite[Theorem 1.2]{Biswas2022}, $\pi_n^* K_1$ is $\mathcal{O}(n,n)$-stable.
    Since multiplying the polarisation by a positive integer does not change slope stability, it follows that $\pi_n^* K_1$ is also $\mathcal{O}(1,1)$-stable.

    We then prove infinitesimal rigidity in the case $n=1$. 
    The bundle $K_1$ has $c_1(K)=(-4,-4)$ and $c_2(K)=12$.
    The expected dimension of the moduli space of bundles with the same rank and Chern classes as $\pi^*(K_1)$ is $\text{expdim}=2 \cdot 6 \cdot 12-2 \cdot 2 (16+16)-8 \cdot 2=0$.
    By \cref{theorem:single-point-moduli} we have that the moduli space consists of one point only.
    Obstruction spaces on moduli spaces of semi-stable sheaves on K3 surfaces vanish (see \cite[p.168]{Huybrechts2016}), therefore, $\pi^*K_1$ is also infinitesimally rigid.
\end{proof}

The same computer-assisted proof method as in the other propositions also works in this case.
This proof is given in \cref{section:appendix-alternative-proofs}, alongside another pen-and-paper proof that stands alone and does not use the results from \cite{Biswas2019,Biswas2022}.

\subsubsection{A second family of rank \texorpdfstring{$3$}{3} examples}

Another rank 3 example is given by the following construction.
Define the rank $3$ vector bundle $K$ over $\mathbb{P}^1 \times \mathbb{P}^1$ as follows:
\begin{align}
\label{equation:K-P1xP1-def2}
K:= \Ker \begin{pmatrix}
    x_1^n \\
    x_2^n \\
    y_1^n \\
    y_2^n
\end{pmatrix},
\quad
0 \rightarrow 
K \rightarrow 
\mathcal{O}(-n,0)^{\oplus 2}\oplus \O(0,-n)^{\oplus 2}
\rightarrow \O(0,0)
\rightarrow
0.
\end{align}

\begin{proposition}
    \label{proposition:another-rank-3-example}
    The bundle $K \rightarrow \mathbb{P}^1 \times \mathbb{P}^1$ defined in \cref{equation:K-P1xP1-def2} is real and stable for $n=2$.
\end{proposition}

\begin{proof}
The bundle $K$ has a real structure by \cref{proposition:real-structure}.
We can write $K=\Im (a)$ for
\begin{align}
\begin{split}
a&:
\mathcal{O}(-2n,0) 
\oplus 
\mathcal{O}(-n,-n) 
\oplus 
\mathcal{O}(-n,-n) 
\oplus 
\mathcal{O}(0,-2n) 
\oplus 
\O(-n,-n)
\oplus 
\O(-n,-n)
\rightarrow
\mathcal{O}(-n,0)^{\oplus 2}\oplus \O(0,-n)^{\oplus 2}
\\
a&=
\begin{pmatrix}
    -x_2^n &0& -y_1^n & 0 & 0&-y_2^n \\
    x_1^n &-y_1^n& 0 & 0 & -y_2^n&0 \\
    0 & x_2^n & x_1^n&-y_2^n &0& 0 \\
    0 & 0 & 0&y_1^n &x_2^n& x_1^n
\end{pmatrix}.
\end{split}
\end{align}

Under the correspondence of modules and sheaves, this corresponds to the module $\Im (a)$ for 
\begin{align*}
a:
R_{\bullet+(2n,0)}
\oplus
R_{\bullet+(n,n)}
\oplus
R_{\bullet+(n,n)}
\oplus
R_{\bullet+(0,2n)} 
\oplus 
R_{\bullet+(n,n)}
\oplus
R_{\bullet+(n,n)}
\rightarrow
R_{\bullet+(n,0)}^2\oplus R_{\bullet+(0,n)}^2
.
\end{align*}

We check stability of $K$ using the Hoppe criterion (\cref{theorem:hoppe-criterion}).
To this end, let first $s=1$.
We have $c_1(K)=(-2n,-2n)$.
Thus we must check $H^0(K(k,l))=0$ for all $k,l \in \mathbb{Z}$ with $k+l \leq 4n/3$, i.e. $k+l \leq \frac{8}{3}$ in the $n=2$ case.

\begin{itemize}
    \item 
    \textbf{Computation of finitely many cohomology groups.}
    We compute using \texttt{Macaulay2} that $H^0(K(k,l))=0$ for $(k,l) \in \{(0,2),(2,0),(1,1),(1,0), (0,1)\}$.

    \item 
    \textbf{Computation of the remaining cohomology groups.}
    It remains to check the case where $l\leq -1$ or $k\leq -1$. 
    Without loss of generality we assume $k\leq -1$ and as before fix $z=[0:1]\in\mathbb{P}^1$ and compute 
    \begin{align*}
        H^0(K|_{\mathbb{P}^1\times \{z\}}\otimes \O(-1))=0.
    \end{align*}
    which as in \cref{proposition:K-on-P1xP1-stable} implies that $H^0(K(k,l))=0$ for $k+l\leq 8/3$.
\end{itemize}

Now for $s=2$ we need that $H^0(\Lambda^2K(k,l))=0$ for $k+l\leq 16/3$.

\begin{itemize}
    \item 
    \textbf{Computation of finitely many cohomology groups.}
    Using $\texttt{Macaulay2}$ we compute
    \begin{align*}
    H^0((\wedge^2 K)(k,l))_{0 \leq k,l \leq 5}
    &=
    \begin{pmatrix}
        0&1&2&13&24&35\\
        0&0&0&8&16&24\\
       0&0& 0&4&8&13\\
        0&0&0&0&0&2\\
        0&0&0&0&0&1\\
        0&0&0&0&0&0
    \end{pmatrix}.
    \end{align*}

    \item 
    \textbf{Computation of the remaining cohomology groups.}
    It only remains to check $H^0(\Lambda^2K(k,l)) = 0$ when $k\leq -1$ or $l\leq -1$. 
    As before we check that for $z=[0:1]$
    \begin{align*}
     H^0(\Lambda^2K|_{\mathbb{P}^1\times \{z\}}\otimes \O(-1))=0.
    \end{align*}
    Thus we have that $H^0(\Lambda^2K(k,l))=0$ for $k+l\leq 16/3$.
    \qedhere
\end{itemize}
\end{proof}

\begin{remark}
    We found the above proof to work for every value of $n$ we tried, for example for $n=1,2,3,4$, so \cref{proposition:another-rank-3-example} also holds at least for these cases.
    However, in the step \emph{Computation of finitely many cohomology groups} we require that $n$ be some concrete value and we have no proof that the proposition holds for \emph{every} $n \geq 1$.
\end{remark}

\Cref{proposition:K-on-P1xP1-stable,proposition:another-rank-3-example} prove the second point of \cref{theorem:main-theorem}.

\subsection{An example on a quartic}
\label{subsection:example-on-quartic}

Consider a K3 surface $X$ with holomorphic involution $\iota: X \rightarrow X$ and anti-holomorphic involution $\sigma: X \rightarrow X$.
In the two subsections preceding this subsection, we constructed $\iota$-invariant bundles on $X$ by pulling back bundles from $X/\< \iota \>$.
Alternatively, one may do the following:
let $E \rightarrow X$ be any bundle on $X$,
then the bundle $\iota^* E \oplus E$ is $\iota$-invariant, because a canonical lift of $\iota$ to $\iota^* E \oplus E$ is given by the following map:
\begin{align*}
    (\iota^* E \oplus E)_x
    =
    E_{\iota(x)} \oplus E_x
    & \rightarrow
    (\iota^* E \oplus E)_{\iota(x)}
    =
    E_{x} \oplus E_{\iota(x)}
    \\
    (v,w) & \mapsto (w,v).
\end{align*}
The bundle $\iota^* E \oplus E$ is never stable, but it is polystable if $E$ is stable.
Furthermore, if $E$ has a real structure, then so has $\iota^* E \oplus E$.
We are thus motivated to construct a real stable bundle on $X$ that is not $\iota$-invariant.

To this end, let $X \subset \mathbb{P}^3(x,y,z,w)$ be the K3 surface from \cite[Section 6]{FNP24}, defined as:
\[
X=Z(-x(x+z-w)(xw - yz)+z(x+z)(xy - z^2)+(xy + w^2)(y^2 - zw)).
\]
It has Picard rank $2$, which is the smallest rank possible for a quartic surface in $\mathbb{P}^3$ admitting an involution, see \cite[Corollary 15.2.12]{Huybrechts2016}.
This and other properties of $X$ are collected in the following lemma:

\begin{lemma}\cite[Lemma 6.2, Proposition 6.3]{FNP24}
    \label{lemma:small-pic-k3-explicit-lemma}
    The surface $X_4$ contains the following curves:
    $$
    \begin{array}{cc}
       C\colon
    \begin{cases}
        xw - yz =0\\
        xy^2 + x^2z - z^3 - xyw + yw^2 - w^3  = 0\\
        xz^2 + z^3 + xyw + w^3 = 0
    \end{cases}
    \;\;\textrm{ and }\;\;  &  D\colon
    \begin{cases}
        xw - yz&=0\\
        xy - z^2&=0\\
        y^2 - zw&=0
    \end{cases}
    \;\;.
    \end{array}
    $$
    The curve $C$ has degree $5$ and genus $2$;
    the curve $D$ is a twisted cubic, hence it has degree $3$ and genus $0$.

    The divisor class group~$\Cl (X)$, is generated by the class of~$C$ and the hyperplane section and is isometric to the lattice $[4\; 5\; 2]$.
\end{lemma}

Let $K$ be the rank $2$ vector bundle over $X$ defined via
\begin{align}
\label{equation:def-bundle-on-K3}
0
\rightarrow
K=\Ker \begin{pmatrix} x\\y\\w \end{pmatrix}
\rightarrow
\mathcal{O}(-1)^{\oplus 3}
\rightarrow
\mathcal{O}
\rightarrow
0.
\end{align}

\begin{proposition}
    The sheaf $K \rightarrow X$ defined in \cref{equation:def-bundle-on-K3} is a stable rank two vector bundle and has a real structure.
\end{proposition}

\begin{proof}

The rank of the matrix $(x,y,w)$ is equal to one everywhere on $X$, as there is no point of the form $[0:0:z:0] \in X$, so $K$ is a  vector bundle of rank $2$.

The sequence \cref{equation:def-bundle-on-K3} defines a vector bundle over $\mathbb{P}^3 \setminus \{[0:0:1:0]\}$ and \cref{proposition:real-structure} gives it a real structure.
The restriction to $X$ inherits this real structure, thus $K \rightarrow X$ is a real bundle.

To check stability, we again apply the Hoppe criterion \cref{theorem:hoppe-criterion}.

We find $c_1(K)=-3 \mathcal{O}(1) \mid_X$ and $\deg(\mathcal{O}(1)|_X)=4$ as well as $\deg(C)=5$ by \cref{lemma:small-pic-k3-explicit-lemma}.
Thus, by \cref{theorem:hoppe-criterion} we have that $K$ is stable if
\[
H^0(X, K \otimes \mathcal{O}(k) \otimes \mathcal{O}(lD))=0
\quad
\text{ for }
\quad
4k+5l \leq \frac{12}{2}=6.
\]

\begin{itemize}
    \item 
    \textbf{Computation of finitely many cohomology groups.}
    A \texttt{Macaulay2} calculation shows that
    \[
    \dim H^0(X, K \otimes \mathcal{O}(k) \otimes \mathcal{O}(lD))
    =0
    \quad
    \text{ for }
    \quad
    (k,l)=(1,0).
    \]

    \item 
    \textbf{Computation of the remaining cohomology groups.}
    Now take $(k,l)$ such that $4k+5l \leq 6$ and $(k,l) \neq (1,0)$.
    Tensoring \cref{equation:def-bundle-on-K3} by $\mathcal{O}(k) \otimes \mathcal{O}(lD)$ and passing to the long exact sequence in cohomology gives
    \begin{align}
    \label{equation:K-long-exact-sequence}
    0
    \rightarrow
    H^0(X, K \otimes \mathcal{O}(k) \otimes \mathcal{O}(lC))
    \rightarrow
    H^0(X, \mathcal{O}(k-1) \otimes \mathcal{O}(lC))^{\oplus 3}.
    \end{align}
    The assumption $4k+5l \leq 6$ and $(k,l) \neq (1,0)$ implies $k \leq 0$ or $l \leq -1$.
    Thus, the divisor $\mathcal{O}(k-1) \otimes \mathcal{O}(lC))$ is not linearly equivalent to an effective divisor, and we have $H^0(X, \mathcal{O}(k-1) \otimes \mathcal{O}(lC))=0$, see \cite[Proposition II.7.7(a)]{Hartshorne2013}.
    Therefore, \cref{equation:K-long-exact-sequence} gives $H^0(X, K \otimes \mathcal{O}(k) \otimes \mathcal{O}(lC))=0$.
    By \cref{theorem:hoppe-criterion}, the bundle $K$ is stable.
    \qedhere
\end{itemize}
\end{proof}

This proves the third part of \cref{theorem:main-theorem}.

\bibliographystyle{abbrv}
\bibliography{lib}

\appendix

\section{Alternative proofs of \texorpdfstring{\cref{proposition:K-on-P1xP1-stable}}{the stability proposition}}
\label{section:appendix-alternative-proofs}

Here, we provide two supplementary 
proofs related to \cref{proposition:K-on-P1xP1-stable}.
The first is a computer-assisted proof which is similar to the proofs of other propositions in the main body of the text and can only prove the stability of $K_n$ for fixed values of $n$.
The second is a pen-and-paper proof which has a similar structure to the computer-assisted proof but carries out the necessary arithmetic checks by hand.
The proof is written for $n=1$ but can be adapted to work for general $n$.
Since both proofs are written for $n=1$, we write $K_1=K$ to ease the notation.

\begin{proof}[Proof of stability for $n=1$ in \cref{proposition:K-on-P1xP1-stable} using computer assistance]
    We check stability of $K$ using the Hoppe criterion \cref{theorem:hoppe-criterion}.
    To this end, let first $s=1$.
    We have $c_1(K)=(-4,-4)$.
    Thus we must check $H^0(K(k,l))=0$ for all $k,l \in \mathbb{Z}$ with $k+l \leq 8/3$.
    
    We can write $K=\Im(a)$ for
    \begin{align}
    \label{equation:K-as-image}
    \begin{split}
    a&:
    \mathcal{O}(-1,-2) 
    \oplus 
    \mathcal{O}(-2,-1) 
    \oplus 
    \mathcal{O}(-1,-2) 
    \oplus 
    \mathcal{O}(-2,-1)
    \rightarrow
    \mathcal{O}(-1,-1)^{\oplus 4}
    \\
    a&=
    \begin{pmatrix}
        -y_2 & -x_2 & 0 & 0 \\
        y_1 & 0 & 0 & -x_2 \\
        0 & x_1 & -y_2 & 0 \\
        0 & 0 & y_1 & x_1
    \end{pmatrix}.
    \end{split}
    \end{align}
    The assignment of bi-graded $R_{\bullet}=\C[x_1,x_2,y_1,y_2]$-modules to coherent sheaves is an exact functor (see e.g. \cite[Proposition II.5.2a]{Hartshorne2013}), and exact functors preserve images.
    Thus, the above presentation of the sheaf $K$ corresponds to the module $\Im(a)$ for the corresponding map
    \[
    a:
    R_{\bullet+(1,2)}
    \oplus
    R_{\bullet+(2,1)}
    \oplus
    R_{\bullet+(1,2)}
    \oplus
    R_{\bullet+(2,1)}
    \rightarrow
    R_{\bullet+(1,1)}^4.
    \]
    By abuse of notation, we used the same symbol $a$ to denote the map of modules corresponding to the map of sheaves from \cref{equation:K-as-image}.
    We also used the notation $R_{\bullet+(k,l)}$ for $k,l \in \Z$ to denote the graded ring $R_{\bullet}$ with the shifted grading that makes the unit in $R_{\bullet}$ have degree $(k,l)$.
    Note that $R_{\bullet+(k,l)}$ is no longer a graded ring, but still a graded module.
    
    \begin{itemize}
        \item 
        \textbf{Computation of finitely many cohomology groups.}
        Using this commutative algebra formulation we use \texttt{Macaulay2} to compute $h^0(K(k,l))=0$  for $(k,l) \in \{(2,0),(1,1),(0,1),(1,0), (0,0), (0,2)\}$.
        
        \item 
        \textbf{Computation of the remaining cohomology groups.}
        Thus, it remains to check those values of $k,l \in \Z$ for which $k+l \leq 8/3$ and $(k,l) \notin \{(2,0),(1,1),(0,1),(1,0), (0,0), (0,2)\}$.
        In this case, $k \leq -1$ or $l \leq -1$.
        Without loss of generality assume that $k \leq -1$.
        Explicit calculation using \texttt{Macaulay2} shows for $z=[0:1]\in \mathbb{P}^1$
        \begin{align}
            \label{equation:rank3-P1P1-example-H0-restriction}
            H^0(K|_{\mathbb{P}^1 \times \{z\}}\otimes \mathcal{O}(-1))=0\; .
        \end{align}
        
        Thus
        \[
            h^0(K(k,l)|_{\mathbb{P}^1 \times \{z\}})
            =
            h^0(K|_{\mathbb{P}^1 \times \{z\}} \otimes \mathcal{O}(k))
            \leq
            h^0(K|_{\mathbb{P}^1 \times \{z\}}\otimes \mathcal{O}(-1))
            =
            0
            \text{ for }
            k,l \in \mathbb{Z}
            \text{ with }
            k \leq -1\; .
        \]
    \end{itemize}
    
    It remains to check the $s=2$ case, i.e., we must check that $h^0((\wedge^2 K)(k,l))=0$ for $k+l \leq \frac{16}{3}$.
    
    \begin{itemize}
        \item 
        \textbf{Computation of finitely many cohomology groups.}
        Explicit calculation using \texttt{Macaulay2} shows
        \begin{align}
            H^0((\wedge^2 K)(k,l))_{0 \leq k,l \leq 5}
            &=
            \begin{pmatrix}
                0&0&2&12&22&32\\
                0&0&1&8&15&22\\
               0&0& 0&4&8&12\\
                0&0&0&0&1&2\\
                0&0&0&0&0&0\\
                0&0&0&0&0&0
            \end{pmatrix}.
        \end{align}
    
        \item 
        \textbf{Computation of the remaining cohomology groups.}
        Thus, it remains to check the values of $k$ and $l$ where $k+l \leq \frac{16}{3}$ and  $k\leq -1$ or $l\leq -1$.
        Assume without loss of generality that $k \leq -1$.
        As before, an explicit calculation using \texttt{Macaulay2} shows for $z=[0:1]$:
        \begin{align}
            H^0(\wedge^2 K|_{\mathbb{P}^1 \times \{z\}}(-1))=0, 
        \end{align}
        and we deduce that $H^0((\wedge^2 K)(k,l))=0$ for $k+l \leq 16/3$ for $k,l \in \mathbb{Z}$.
    \end{itemize}
\end{proof}

\begin{proof}[Proof of stability for $n=1$ in \cref{proposition:K-on-P1xP1-stable} without computer assistance]
    Let $s:S=\P^1 \times \P^1 \rightarrow Q \subset \mathbb{P}^3$ be the Segre embedding.
    Pulling back the Euler sequence gives $K \simeq s^*(\Omega_{\mathbb{P}^3}|_Q)$.
    The conormal sequence \cite[Lemma 29.34.17]{stacks-project} of $Q$ together with $\mathcal{C}_{Q/\mathbb{P}^3} \simeq \mathcal{O}_Q(-2)$ from \cite[Lemma 31.14.2]{stacks-project} gives
    \[
    0 
    \rightarrow \mathcal{O}_Q(-2)
    \rightarrow \Omega_{\mathbb{P}^3}|_Q
    \rightarrow \Omega_Q
    \rightarrow 0.
    \]
    Pulling back under $s$ gives
    \begin{align}
    \label{equation:K-SES}
    0\rightarrow
    \mathcal{O}(-2,-2)
    \rightarrow K
    \rightarrow \mathcal{O}(-2,0) \oplus \mathcal{O}(0,-2)
    \rightarrow 0.
    \end{align}
    Let the extension class of this sequence be
    $(\alpha,\beta) \in Ext^1(\mathcal{O}(-2,0),\mathcal{O}(-2,-2)) \oplus Ext^1(\mathcal{O}(0,-2),\mathcal{O}(-2,-2)) = Ext^1(\mathcal{O}(-2,0) \oplus \mathcal{O}(0,-2),\mathcal{O}(-2,-2))$ (we used \cite[Prop. 3.3.4]{weibel1994introduction} for the Ext groups of a sum).
    
    Claim:
    $\alpha \neq 0$ and $\beta \neq 0$.
    
    Proof of claim:
    Let $F=\P^1 \times \{p\}$ for some $p \in \P^1$.
    Then $L=s(F)$ is a line in $\P^3$, so
    \begin{align}
    \label{equation:K-line-splitting-type}
    K|_F
    \simeq
    \Omega_{\P^3}|_{L}
    \simeq
    \mathcal{O}_{\P^1}(-2) \oplus \mathcal{O}_{\P^1}(-1)^{\oplus 2},
    \end{align}
    where in the second step we used the conormal sequence for $L \subset \P^3$ and got
    \[
    0
    \rightarrow I_L/I_L^2 \simeq \mathcal{O}_L(-1)^{\oplus 2}
    \rightarrow \Omega_{\P^3}|_L
    \rightarrow \mathcal{O}(-2)
    \rightarrow 0,
    \]
    which is a split extension due to $Ext^1(\mathcal{O}(-2),\mathcal{O}(-1)^{\oplus 2})=H^1(\mathcal{O}(1))^{\oplus 2}=0$.
    Restricting \cref{equation:K-SES} to $F$ gives
    \[
    0 \rightarrow
    \mathcal{O}(-2) \rightarrow
    K|_F \rightarrow
    \mathcal{O} \oplus \mathcal{O}(-2)
    \rightarrow 0.
    \]
    If $\beta=0$, then this splits and gives a trivial summand in $K|_F$, which is a contradiction to \cref{equation:K-line-splitting-type}.
    Hence, $\beta \neq 0$, and analogously one finds $\alpha \neq 0$.
    (End of proof of claim.)
    
    We have $c_1(K)=(-4,-4)$ and must rule out destabilising subsheaves $F \subset K$ with $\mu_H(F) \geq \mu_H(K)=-\frac{8}{3}$.
    By \cite[Prop. 1.2.6]{Huybrechts2010} it suffices to check $F$ so that $K/F$ is torsion-free.
    By \cite[Lemma 31.12.7]{stacks-project}, $F$ is reflexive.
    Because $S$ is smooth of dimension $2$, by \cite[Lemma 31.12.15]{stacks-project}, $F$ is locally free.
    
    First check $F$ with $rk(F)=1$, i.e. $F \simeq \mathcal{O}(a,b)$.
    Since $F \hookrightarrow \mathcal{O}(-1,-1)^{\oplus 4}$ we have $a \leq -1$, $b \leq -1$.
    On the other hand, $a+b=\mu_H(F) \geq -\frac{8}{3}$, so $a+b \geq -2$.
    Hence, $(a,b)=(-1,-1)$.
    But twisting the defining sequence of $K$ by $\mathcal{O}(1,1)$ gives
    \[
    0 \rightarrow K(1,1) \rightarrow \mathcal{O}^{\oplus 4} \rightarrow \mathcal{O}(1,1) \rightarrow 0,
    \]
    where the map $\mathcal{O}^{\oplus 4} \rightarrow \mathcal{O}(1,1), (a,b,c,d) \mapsto a \cdot x_1 y_1+b \cdot x_1 y_2+c \cdot x_2 y_1+d \cdot x_2 y_2$ induces an isomorphism $H^0(\mathcal{O}^{\oplus 4}) \rightarrow H^0(\mathcal{O}(1,1))$.
    Hence, from the long exact sequence in cohomology, we get $H^0(K(1,1))=0$ and therefore $\Hom(\mathcal{O}(-1,-1),K) \simeq H^0(K \otimes \mathcal{O}(1,1)) = H^0(K(1,1))=0$.
    Thus, $F$ cannot inject into $K$ and we checked the rank $1$ case.
    
    Now assume $rk(F)=2$ and write $\det F=\mathcal{O}(a,b)$.
    As before, $F$ is locally free, so $F \rightarrow K \rightarrow \mathcal{O}(-1,-1)^{\oplus 4}$ induces an injection $\det F=\Lambda^2 F \rightarrow \Lambda^2(\mathcal{O}(-1,-1)^{\oplus 4}) \simeq \mathcal{O}(-2,-2)^{\oplus 6}$, hence $a \leq -2$, $b \leq -2$.
    From $\frac{a+b}{2} = \mu_H(F) \geq -\frac{8}{3}$ we have $a+b \geq -5$, so $(a,b) \in \{(-2,-2),(-2,-3),(-3,-2)\}$.
    It follows that it suffices to prove $H^0(\Lambda^2 K(2,2))=H^0(\Lambda^2 K(2,3))=H^0(\Lambda^2 K(3,2))=0$.
    
    We have $\Lambda^2 K \simeq K^* \otimes \det K \simeq K^* (-4,-4)$, because $\det K \simeq \det (\mathcal{O}(-1,-1)^{\oplus 4}) \otimes \det (\mathcal{O})^{-1} \simeq \mathcal{O}(-4,-4)$ from the defining sequence of $K$.
    Thus, the three $H^0$ vanishing conditions become:
    $H^0(K^*(-2,-2))=H^0(K^*(-2,-1))=H^0(K^*(-1,-2))=0$.
    The dual of \cref{equation:K-SES} twisted by $\mathcal{O}(-2,-2)$ is
    \[
    0 \rightarrow \mathcal{O}(0,-2) \oplus \mathcal{O}(-2,0) \rightarrow K^*(-2,-2) \rightarrow \mathcal{O} \rightarrow 0.
    \]
    In the long exact sequence in cohomology, the left hand side has no global sections, and the connecting morphism $H^0(\mathcal{O}) \rightarrow H^1(\mathcal{O}(0,-2)) \oplus H^1(\mathcal{O}(-2,0))$ maps $1 \mapsto e'$, where $e'$ is the extension class of the dual twisted sequence by \cite[Lemma 22.2]{Swanson2018}.
    Non-splitting is preserved under twists and taking duals, so it is non-zero, because the extension class of \cref{equation:K-SES} is non-zero, which was shown above.
    Thus, $H^0(K^*(-2,-2))=0$.
    
    To see that $H^0(K^*(-2,-1))=0$ note that by Serre duality:
    \[
    H^0(K^*(-2,-1))^*
    \simeq
    H^2((K(2,1) \otimes \mathcal{O}(-2,-2))
    =
    H^2(K(0,-1)).
    \]
    By the Künneth formula:
    \begin{align*}
    H^1(\mathcal{O}(0,-1))
    &=
    H^0(\mathcal{O}) \otimes H^1(\mathcal{O}(-1))
    \oplus
    H^1(\mathcal{O}) \otimes H^0(\mathcal{O}(-1))
    =0,
    \\
    H^2(\mathcal{O}(-1,-2))
    &=
    H^1(\mathcal{O}(-1)) \otimes H^1(\mathcal{O}(-2))=0.
    \end{align*}
    Twisting the defining sequence for $K$ by $\mathcal{O}(0,-1)$ and taking the long exact sequence in cohomology we get
    \[
    \dots 
    \rightarrow
    H^1(\mathcal{O}(0,-1))
    \rightarrow
    H^2(K(0,-1))
    \rightarrow
    H^2(\mathcal{O}(-1,-2)^{\oplus 4})
    \rightarrow \dots
    \]
    Hence $H^2(K(0,-1))=0$, which proves $H^0(K^*(-2,-1))=0$.
    Analogously, one checks $H^0(K^*(-1,-2))=0$, which finishes the proof.
\end{proof}

\noindent                               

\vskip 8pt

\noindent{\small\sc Dipartimento di Matematica e Applicazioni `Renato Caccioppoli',  Università degli studi di Napoli `Federico II', Via Cintia, Monte S. Angelo I-80126 Napoli, Italy}

\noindent E-mail: {\tt \href{mailto:dino.festi@unina.it}{dino.festi@unina.it}}

\vskip 8pt

\noindent{\small\sc Imperial College London, Department of Mathematics, 180 Queen's Gate, South Kensington, London SW7 2RH, the United Kingdom}

\noindent E-mail: {\tt \href{mailto:d.platt@imperial.ac.uk}{d.platt@imperial.ac.uk}}

\vskip 8pt

\noindent{\small\sc University of Münster, Einsteinstraße 62, 48149 Münster, Germany}

\noindent E-mail: {\tt \href{mailto:rsinghal@uni-muenster.de}{rsinghal@uni-muenster.de}}

\vskip 8pt

\noindent{\small\sc Beijing Institute of Mathematical Sciences and Applications (BIMSA), No. 544, Hefangkou Village, Huaibei Town, Huairou District, Beijing 101408, China}

\noindent E-mail: {\tt \href{mailto:ytanaka@bimsa.cn}{ytanaka@bimsa.cn}}

\end{document}